\documentclass[a4paper,12pt]{article}

\usepackage{amssymb}
\usepackage{amsmath}
\usepackage{amsfonts}
\usepackage[pdftex,pdfstartview=FitH]{hyperref}
\usepackage{graphicx}
\usepackage{color}
\usepackage{eurosym}
\usepackage[utf8]{inputenc}
\usepackage[english]{babel}
\usepackage[mathscr]{eucal}
\usepackage{latexsym, amsmath, amsfonts, amscd, amssymb, verbatim,color,cite}
\usepackage[pdftex,pdfstartview=FitH]{hyperref}
\hypersetup{colorlinks=true,linkcolor=blue,citecolor=blue,urlcolor=blue}

\def\tto{\;{\lower 1pt \hbox{$\rightarrow$}}\kern -10pt
\hbox{\raise 2pt \hbox{$\rightarrow$}}\;}

\def\R{\mathbb R}

\hypersetup{colorlinks=false,linkbordercolor={1 1 1},citebordercolor={1 1 1}}
\newtheorem{theorem}{Theorem}[section]

\newtheorem{corollary}[theorem]{Corollary}

\newtheorem{definition}[theorem]{Definition}
\newtheorem{example}[theorem]{Example}

\newtheorem{lemma}[theorem]{Lemma}

\newtheorem{proposition}[theorem]{Proposition}
\newtheorem{remark}[theorem]{Remark}

\newenvironment{proof}[1][Proof]{\noindent\textbf{#1.} }{\hfill$\square$}
\title{ \sf \Large A  Bounded Degree Lasserre Hierarchy  with SOCP Relaxations for  Global Polynomial Optimization and  Applications\thanks{Research was  supported by a research grant from Australian Research Council.}\\
\medskip
}
\author{T.D. Chuong\thanks{%
Department of Applied Mathematics, University of New South Wales, Sydney
2052, Australia and Saigon University, Vietnam. Email: chuongthaidoan@yahoo.com}, \ V. Jeyakumar\thanks{%
Department of Applied Mathematics, University of New South Wales, Sydney
2052, Australia. Email: v.jeyakumar@unsw.edu.au } \ and G. Li \thanks{%
Department of Applied Mathematics, University of New South Wales, Sydney
2052, Australia. Email: g.li@unsw.edu.au. }}

\begin{document}

\maketitle

\vspace{-1.2cm}
\begin{abstract}
In this paper, we propose a new convergent conic programming hierarchy of relaxations involving both semi-definite cone and second-order cone constraints for solving nonconvex polynomial optimization problems to global optimality. The significance of this hierarchy is that
 the size and number of the semi-definite and second-order cone constraints of the relaxations are fixed and independent of the step or level of the approximation in the hierarchy. Using the Krivine-Stengle's certificate of positivity in real algebraic geometry, we establish the convergence of the hierarchy of relaxations, extending the very recent so-called bounded degree Lasserre hierarchy \cite{Bounded_SOS}. In particular, we also provide a convergent bounded degree second-order cone programming (SOCP) hierarchy for solving polynomial optimization problems.  We then present finite convergence at step one of the SOCP hierarchy for two classes of polynomial optimization problems: a subclass of convex polynomial optimization problems where the objective and constraint functions are SOCP-convex polynomials, defined in terms of specially structured sum of squares polynomials, and a class of polynomial optimization problems, involving polynomials with essentially non-positive coefficients. In the case of one-step convergence for problems with SOCP-convex polynomials, we show how a global solution is recovered from the relaxation via Jensen's inequality of  SOCP-convex polynomials. As an application, we derive a corresponding convergent conic linear programming hierarchy for conic-convex semi-algebraic programs. Whenever the semi-algebraic set of the conic-convex program is described  by convex polynomial inequalities, we show further that the values of the relaxation problems converge to the common value of the convex program and its Lagrangian dual under a constraint qualification.
 \medskip

 \noindent\textbf{Key words}: Polynomial optimization, conic programming relaxations, global optimization, cone-convex polynomial programs.

\end{abstract}

\section{Introduction}

The sums-of-squares (SOS) based hierarchy of semi-definite programming relaxations, often  {referred} to as Lasserre hierarchy, has proved to be a powerful approach for examining and solving polynomial optimization problems to global optimality using semi-definite programming. The convergence of the
Lasserre hierarchy relies on a fundamental sum of squares representation of
positivity of a polynomial over semialgebraic sets established by Putinar \cite{Lassere_book}. Unfortunately, this scheme applies to solve polynomial problems of modest size only because as the dimension of the problem  or the level of the approximation increases, the corresponding computational cost for solving the associated semi-definite programming problems in the hierarchy grows significantly.

Nowadays, a great deal of research in the area of polynomial optimization has been devoted to reducing the computational burden associated with solving the relaxation problems of the hierarchy. This is often done using one or more of the following approaches, where the scheme: (i) exploits special structures to restrict sum-of-squares in the hierarchy, such as sparsity, symmetry, of the underlying polynomial optimization problem to improve performance \cite{sparse-lasserre,JKLL,Waki}.  The  Lasserre hierarchy has recently been shown to solve industrial-scale optimal power flow problems in electrical engineering with several thousands of variables and constraints \cite{large-scale} using restricted sum-of squares in the hierarchy; or  (ii) employs alternative conic programming hierarchies with less computational cost such as the scaled diagonally dominant sum of squares (SDSOS) hierarchies \cite{A,preprint,sdp-socp} or (iii) restricts the degree of the sum-of-squares of the hierarchy by using a different representations of positivity to the Putinar representation, such as the Krivine-Stengle's certificate of positivity in real algebraic geometry \cite{Bounded_SOS,sparse-lasserre}. This approach proposed a bounded degree hierarchy of semi-definite programming (SDP) relaxations where
the size of the semidefinite matrix involved in the hierarchy, in contrast to the standard Lasserre hierarchy, is fixed.

The present work was motivated by the desire to develop a convergent bounded degree hierarchy involving both SDP and SOCP relaxations using approaches (ii) and (iii) above concurrently and to present explicit classes of problems for which finite convergence is achieved and global solutions are found by solving a single relaxation problem.

%

Our contributions to global polynomial optimization are itemized below.
\begin{description}
 \item[(i)]  We propose a bounded degree hierarchy of conic programming relaxations involving both semi-definite and
 second-order cone constraints for solving nonconvex polynomial optimization problems to global optimality, where
 the size and number of the semi-definite and second-order cone constraints of the relaxations are fixed and independent of the step or level of the approximation in the hierarchy. Using the Krivine-Stengle's certificate of positivity in real algebraic geometry \cite{Lassere_book}, we establish convergence of the hierarchy of relaxations, extending the very recent so-called bounded degree Lasserre hierarchy \cite{Bounded_SOS}. In particular, we show that the values of both primal and dual relaxations converge to the value of the original polynomial optimization problem and that the Lagrangian duality also holds between the primal and dual relaxation problems under a constraint qualification.

 \item[(ii)]  As an important special case, we obtain a convergent bounded degree second-order cone programming (SOCP) hierarchy for solving polynomial optimization problems. We  then present finite convergence at step one of the SOCP hierarchy for two classes of polynomial optimization problems: a subclass of convex polynomial optimization problems where the objective and constraint functions are SOCP-convex polynomials, defined in terms of SDSOS polynomials  {\cite{A,preprint}}, and a class of polynomial optimization problems, involving polynomials with essentially non-positive coefficients. We do this by establishing new representation results for nonnegative  polynomials in terms of SDSOS polynomials.

\item[(iii)]  By obtaining a version of Jensen's inequality for SOCP-convex polynomials, we show how global solution is recovered from a relaxation problem of an SOCP-convex polynomial programs whenever it has a finite convergence at one-step of the SOCP hierarchy.   A key feature of  SOCP-convexity is that checking whether a given polynomial is SOCP-convex or not can be done by solving a second-order cone programming problem.

 \item[(iv)] As an application, we present a convergent conic linear programming hierarchy for  a broad class of conic-convex semi-algebraic programs of the form
 \begin{align}
  \inf_{x\in \R^{n}}{\{f(x) \mid  x\in K, \; G(x)\in S\}}, \tag{CP}
\end{align}
  \vspace{-0.6cm}

where $f$ is a convex polynomial on $\mathbb{R}^n$, $S$ is a closed convex cone of $\mathbb{R}^p$, $K$ is a basic semi-algebraic set given by $K:=\{x\in\R^n\mid  g_i(x)\ge 0,\; i=1,\ldots, m\}$ for some (not necessarily convex) polynomials $g_i, i=1,\ldots, m,$ on $\mathbb{R}^n$   and  $G:\R^n\to \mathbb{R}^p$ is an $S$-concave polynomial. We achieve this with the help of both the hyperplane separation of convex geometry \cite{Mor-Nam-14} and Krivine-Stengle's certificate of positivity of algebraic geometry. In the case where the convex semi-algebraic set $K$ is described by convex polynomials $g_i$'s we show further that the values of the relaxation problems of $(CP)$ converge to the common value of $(CP)$ and its Lagrangian dual problem whenever a suitable constraint qualification holds. Related duality results for nonconvex polynomial programs with polynomial multipliers may be found in \cite{jeya-chuong-18}.
\end{description}

The organization of this paper is as follows. In Section 2, we first collect basic notions and preliminaries on polynomials and the scaled diagonally dominant sums-of-squares polynomials.
In Section 3, we introduce the bounded degree hierarchy and establish its asymptotic convergence. In Section 4, we establish
new representation results for nonnegative polynomials in terms of scaled diagonally dominant sums-of-squares polynomials, and obtain one-step convergence results for two classes of polynomial optimization problems. In Section 5 we present results that show how global solutions of SOCP-convex polynomial programs can be found from SOCP relaxations whenever they enjoy one-step convergence. Section 6 {establishes} convergence results for semi-algebraic cone-convex polynomial programs.

\vspace{-0.6cm}
\section{Preliminaries}

First of all, let us recall some  notations and basic facts on  sums-of-squares polynomial and semi-definite programming problems.  Recall that $S^{n}$ denotes the space of symmetric $(n\times n)$ matrices
with the
trace inner product and $\succeq $ denotes the L\"{o}wner partial order of $
S^{n}$, that is, for $M,N\in S^{n},$ $M\succeq N$ if and only if
$(M-N)$ is positive semidefinite. The set consisting of all positive semidefinite $(n \times n)$ matrices is denoted by $S^n_+$.

Consider a polynomial $f$ with degree at most $d$ where $d$ is an even number, and let $l:=d/2$. Let $\mathbb{R}_d[x_1,\ldots,x_n]$ (or $\mathbb{R}[x]$) be the space consisting of all real polynomials on $\mathbb{R}^n$ with degree $d$ and let $s(d,n)$ be the dimension of $\mathbb{R}_d[x_1,\ldots,x_n]$. Write the canonical basis of $\mathbb{R}_d[x_1,\ldots,x_n]$ by
$$ x^{(d)}:=(1,x_1,x_2,\ldots,x_n,x_1^2,x_1x_2,\ldots,x_2^2,\ldots,x_n^2,\ldots,x_1^{d},\ldots,x_n^d)^T.$$
For each $1 \le \alpha \le s(d,n)$, we denote $i(\alpha)=(i_1(\alpha),\ldots,i_n(\alpha)) \in (\mathbb{N}_0)^n$ to be the multi-index such that
 $$x^{(d)}_{\alpha}=x^{i(\alpha)}:=x_1^{i_1(\alpha)}\ldots x_n^{i_n(\alpha)},$$ { where $\mathbb{N}_0:= \mathbb{N} \cup \{0\}$.}
Let the monomials $m_\alpha(x)=x^{(d)}_{\alpha}$ be the $\alpha$-th coordinate of $x^{(d)}$, $1 \leq \alpha \leq s(d,n)$. Thus, we can write
\begin{align}\label{bs-wrote} f(x) = \sum\limits_{\alpha=1}^{s(d,n)} f_{\alpha}m_{\alpha}(x)=\sum\limits_{\alpha=1}^{s(d,n)} f_{\alpha}x^{(d)}_{\alpha}.\end{align}   

 We say that a real polynomial $f$ is sums-of-squares (cf. \cite{Lassere_book}) if there exist real polynomials $f_j$, $j=1,\ldots,q$, such that $f=\sum_{j=1}^qf_j^2$. The set consisting of all sum of squares real polynomials (resp. the set consisting of all sum of squares real polynomials with degree at most $d$) is denoted by $\Sigma^2[x]$ or $\Sigma^2_n$  (resp. $\Sigma^2_d[x]$ or $\Sigma^2_{n,d}$).  Let $d$ be an even number and let $l=d/2$. Then, $f$ is a sum-of-squares polynomial with degree $d$ if and only if there exists a positive semi-definite symmetric $(s(l,n) \times s(l,n))$ matrix $Q$  such that
\begin{equation}\label{eq:useful}
f(x)=(x^{(l)})^TQ x^{(l)}.
\end{equation}
 Then, by comparing the coefficients in (\ref{eq:useful}), we have the following linear matrix inequality characterization of a sum-of-squares polynomial.
\begin{lemma}\label{th:sos} Let $d$ be an even number.
For a polynomial $f$ on $\mathbb{R}^n$ with degree at most $d$, $f$ is sum-of-squares  if and only if the following linear matrix inequality problem has a solution
 \begin{eqnarray*}
\left\{ \begin{array}{l}
Q \in S_+^{s(l,n)} \\
 \displaystyle f_{\alpha}=\sum_{1 \le \beta,\gamma \le s(l,n), i(\beta)+i(\gamma)=i(\alpha)} Q_{\beta \gamma}, \ 1  \le \alpha \le s(d,n),\ l=d/2.
        \end{array}\right.
\end{eqnarray*}
\end{lemma}

\begin{definition}{\bf (Scaled diagonally dominant sums-of-squares polynomials \cite{A})}
We say a polynomial $f$ with degree $d=2l$ is scaled diagonally dominant sums-of-squares (SDSOS) if there exist $p \in \mathbb{N}$ with $1 \le p \le s(l,n)$ and nonnegative number $\alpha_i$, $\beta_{ij}^+,\beta_{ij}^-$ and $\gamma_{ij}^+,\gamma^-_{ij}$ such that
\[
 f(x)=\sum_{i=1}^{p} \alpha_i m_i^2(x)+\sum_{i,j=1, i \neq j}^{p} (\beta_{ij}^+ m_i(x)+\gamma_{ij}^+ m_j(x))^2+\sum_{i,j=1,  i \neq j}^{p} (\beta_{ij}^- m_i(x)-\gamma_{ij}^- m_j(x))^2;
 \]
 where $m_i$ and $m_j$ are monomials in the variable $x$. We denote the set ${\bf SDSOS}$ or  ${\bf SDSOS}_n$  (resp.  ${\bf SDSOS}_{n,d}$) as the set consisting of all scaled diagonally dominant sums-of-squares polynomials  (resp. the set consisting of all scaled diagonally dominant sums-of-squares polynomials 	with degree at most $d$). Meanwhile,  the set ${\bf SDSOS}_m$ (resp.  ${\bf SDSOS}_{m,d}$) stands for the set consisting of all scaled diagonally dominant sums-of-squares polynomials  (resp. the set consisting of all scaled diagonally dominant sums-of-squares polynomials 	with degree at most $d$ with respect to $(x_1,\ldots, x_m),$ where $1\le m\le n.$
 \end{definition}
 It is clear that any scaled diagonally dominant sums-of-squares polynomial is sums-of-squares. On the other hand, the converse is not true in general.
 Moreover,  we note that, in the algebraic geometry literature, this subclass of sums-of-squares polynomials is also called sums-of-binomial-squares polynomials \cite{Fiaco,Marshall}.

 An important and useful feature for a scaled diagonally dominant sums-of-squares polynomial is that checking a given real polynomial is scaled diagonally dominant sums-of-squares can be equivalently
 reformulated as a feasibility problem of a second-order cone programming. To see this, we need the notion of generalized diagonally dominant matrix.  
Recall that an $(n \times n)$ matrix $Q$ is a diagonally dominant matrix  \cite{Matrix} if $$Q_{ii} \ge \sum_{j \neq i} |Q_{ij}| \mbox{ for each }  i=1,\ldots,n.$$
More generally,
an $(n \times n)$ matrix $Q$ is called a generalized diagonally dominant matrix  if there exists a diagonal matrix $D$ with all positive diagonal elements
such that $D^TQD$ is a diagonally dominant  matrix.


Let $n_1, n_2\in\mathbb{N}$,  $0\le n_1, n_2\le n, n=n_1+n_2$ and denote \begin{align}\label{U-n1-n1}U(n_1, n_2):=\{\sigma\in\mathbb{R}[x]\mid \sigma=\sigma_1+\sigma_2, \sigma_1\in \Sigma^2_{n_1}, \sigma_2\in {\bf SDSOS}_{n_2}\},\\U_d(n_1, n_2):=\{\sigma\in\mathbb{R}[x]\mid \sigma=\sigma_1+\sigma_2, \sigma_1\in \Sigma^2_{n_1,d}, \sigma_2\in {\bf SDSOS}_{n_2,d}\}\notag\end{align} where $d\in\mathbb{N}, \Sigma^2_{n_1}=\Sigma^2_{n_1,d}:=\{0\}$ for $n_1=0$ and ${\bf SDSOS}_{n_2}={\bf SDSOS}_{n_2,d}:=\{0\}$ for $n_2=0.$

We say that a polynomial $f$ on $\mathbb{R}^n$ (resp., with degree $d$)  is $(n_1,n_2)$-semidefinite and scaled diagonally dominant sums-of-squares ($(n_1, n_2)$-SDP-SDSOS) if $f\in U(n_1, n_2)$ (resp., $f\in U_d(n_1, n_2).$ The following property provides characterizations of SDP-SDSOS polynomial in terms of mixed semidefinite and generalized diagonally dominant matrix and an associated semidefinite-second-order cone program.

To do this, we further define two notation of multi-index sets. Let $d,n$ be any natural number. Recall that for each $1 \le \alpha \le s(d,n)$, we denote $i(\alpha)=(i_1(\alpha),\ldots,i_n(\alpha)) \in {(\mathbb{N}_0)^n}$ to be the multi-index such that
 $$x^{(d)}_{\alpha}=x^{i(\alpha)}:=x_1^{i_1(\alpha)}\ldots x_n^{i_n(\alpha)}.$$
 Let $n_1, n_2\in\mathbb{N}$, $0\le n_1, n_2\le n, n=n_1+n_2$. For $1\le \alpha \le s(d,n_1)$, we define $$i_+(\alpha)=\{v \in {(\mathbb{N}_0)^n}:
  v=(i(\alpha),\underbrace{0,\ldots,0}_{n_2})\}$$
 and, for $1\le \alpha \le s(d,n_2)$, we define $$i^+(\alpha)=\{v \in {(\mathbb{N}_0)^n}:  v=(\underbrace{0,\ldots,0}_{n_1}, i(\alpha))\}.$$
In what follows, we also 
denote $X_1:=(x_1,\ldots,x_{n_1})$ and $X_2:=(x_{n_1+1},\ldots,x_n)$.

\begin{proposition}{\bf (SDP-SOCP reformulation of SDP-SDSOS polynomials)} \label{prop:1-SDP} Let $n_1, n_2\in\mathbb{N}$, $0\le n_1, n_2\le n, n=n_1+n_2$.
Let $f$ be a real polynomial on $\mathbb{R}^n$ with an even degree $d$, and let $l=d/2$. Then, the following statements are equivalent:
\begin{itemize}
 \item[{\rm (i)}] $f$ is an $(n_1, n_2)$-SDP-SDSOS polynomial.
 \item[{\rm (ii)}] $f(x)=\big(X_1^{(l)}\big)^TQ^1 X_1^{(l)}+\big(X_2^{(l)}\big)^TQ^2 X_2^{(l)}$ for each $x=(X_1,X_2)\in\mathbb{R}^n$, where $Q^1\in S^{n_1}_+$ if $n_1>0$, $Q^1:=0$ if $n_1=0,$ $Q^2$ is an $(s(l,n_2) \times s(l,n_2))$  generalized diagonally dominant matrix if $n_2>0$ and $Q^2:=0$ if $n_2=0$.
 \item[{\rm (iii)}] The following SDP-SOCP  feasibility problem has a solution
 there exist an $(s(l,n_1) \times s(l,n_1))$ symmetric matrix $Q^1$ if $n_1>0$, $Q^1:=0$ if $n_1=0$  and an $(s(l,n_2) \times s(l,n_2))$ symmetric matrix $Q^2$ if $n_2>0$ and $Q^2:=0$ if $n_2=0$ and symmetric matrices $M^{ij}$, $1 \le i, j \le s(l,n_2)$ if $n_2>0$ and $M^{ij}:=0$ if $n_2=0$, such that
 \[
\left\{
\begin{array}{l}
\displaystyle f_{\alpha} = \sum_{\substack{ 1\le \beta \le s(l,n_1), \\ 1\le \gamma \le s(l,n_1) \\
i(\alpha)=i_+(\beta)+i_+(\gamma)}}Q^1_{\beta \gamma}+\sum_{\substack{1\le \beta \le s(l,n_2), \\ 1\le \gamma \le s(l,n_2), \\ i(\alpha)=i^+(\beta)+i^+(\gamma)}}Q^2_{\beta \gamma}, \\
\\Q^1\in S^{s(l,n_1)}_+,\  Q^2=\sum_{1 \le i,j \le s(l,n_2)} M^{ij}, \\
 M^{ij}_{\beta\gamma}=0,\, \forall \ (\beta,\gamma) \notin \{(i,i),(j,j),(i,j),(j,i)\}, \\
 \|\bigg(\begin{array}{c}    2 M^{ij}_{ij}\\
                                                        M^{ij}_{ii}-M^{ij}_{jj}
                                                                                  \end{array}\bigg)\| \le M^{ij}_{ii}+M^{ij}_{jj}, 1 \le i,j \le s(l,n_2).
\end{array}\right.
\]
\end{itemize}
\end{proposition}
\begin{proof}  Since  $f$ is an $(n_1, n_2)$-SDP-SDSOS polynomial, there exist  $f_1\in \Sigma^2_{n_1,d}, f_2\in {\bf SDSOS}_{n_2,d}$ such that    $f(x)=f_1(X_1)+f_2(X_2)$ for all $x=(X_1,X_2)\in\R^n. $ By a characterization of  the sum-of-squares polynomial   in \eqref{eq:useful}, it holds that $f_1(X_1)=\big(X_1^{(l)}\big)^TQ^1 X_1^{(l)}$, where $Q^1\in S^{n_1}_+$ if $n_1>0$, $Q^1:=0$ if $n_1=0.$ Thanks to Theorems~7 and 8 and Lemma~9 in \cite{preprint}, we have $f_2(X_2)=\big(X_2^{(l)}\big)^TQ^2 X_2^{(l)}$, where  $Q^2$ is an $(s(l,n_2) \times s(l,n_2))$  generalized diagonally dominant matrix if $n_2>0$ satisfying  $Q^2=\sum_{1 \le i,j \le s(l,n_2)} M^{ij},
 M^{ij}_{\beta\gamma}=0,\, \forall \ (\beta,\gamma) \notin \{(i,i),(j,j),(i,j),(j,i)\},$
$$ \|\bigg(\begin{array}{c}    2 M^{ij}_{ij}\\
                                                        M^{ij}_{ii}-M^{ij}_{jj}
                                                                                  \end{array}\bigg)\| \le M^{ij}_{ii}+M^{ij}_{jj}, 1 \le i,j \le s(l,n_2)
$$  and $Q^2:=0$ if $n_2=0$.
Note that $f(x)=\big(X_1^{(l)}\big)^TQ^1 X_1^{(l)}+\big(X_2^{(l)}\big)^TQ^2 X_2^{(l)}$. By comparing the coefficients of each monomials, we obtain that
\[f_{\alpha} = \sum_{\substack{ 1\le \beta \le s(l,n_1), \\ 1\le \gamma \le s(l,n_1) \\
i(\alpha)=i_+(\beta)+i_+(\gamma)}}Q^1_{\beta \gamma}+\sum_{\substack{1\le \beta \le s(l,n_2), \\ 1\le \gamma \le s(l,n_2), \\ i(\alpha)=i^+(\beta)+i^+(\gamma)}}Q^2_{\beta \gamma}{,}
\]
Thus, the conclusion follows.
\end{proof}

\section{ Bounded Degree Mixed SDP-SOCP Hierarchies}
Consider the following nonconvex polynomial optimization problem:
\begin{align}\label{P}
 \inf_{x \in \mathbb{R}^n} { \{f(x) \mid   g_i(x) \ge 0, i=1,\ldots,m\}},\tag{P}
\end{align}
where $f$ and $g_i$ are real polynomials with degrees at most $d$. We denote the feasible set of problem \eqref{P} as
$K:=\{x \in \mathbb{R}^n \mid g_i(x) \ge 0, i=1,\ldots,m\}$.

Throughout this section, we need the following assumption.

{\bf Assumption A:} We assume that $K$ is a nonempty compact set and $\{1, g_1,\ldots,g_m\}$ generates $\mathbb{R}[x]$, where $\mathbb{R}[x]$ is the  ring of polynomials in the variable of $x=(x_1,\ldots,x_n)${.}

We note that, Assumption A is automatically satisfied for the box constraints where $K=\{x: 0 \le x_i \le 1, i=1,\ldots,n\}$ because the collection of polynomials $\{1,x_1,\ldots,x_n,1-x_1,\ldots,1-x_n\}$ generates $\mathbb{R}[x]$.
Moreover, as long as the feasible set $K$ is compact, assumption A can also be enforced by adding redundant constraints. Indeed, as $K$ is compact, there exists $M>0$ such that
$x_i \le M$ for all $x \in K$. Then, by adding redundant constraints $x_i \le M, i=1,\ldots,n$ to problem (P), we see that Assumption A holds because
 $\{1, g_1,\ldots,g_m, M-x_1,\ldots,M-x_n\}$ generates $\mathbb{R}[x]$,

For problem (P), let $M$ be a positive number such that $\displaystyle M > \max_{1 \le i\le m}\sup_{x \in K}\{g_i(x)\}$ and denote $\widehat{g}_i(x)=\frac{g_i(x)}{M}$. Let $n_1, n_2\in\mathbb{N}$, $0\le n_1, n_2\le n, n=n_1+n_2$ and fix a positive even number ${r} \in \mathbb{N}$.
 We now define a hierarchy of SDP-SOCP relaxation problem as follows: for all $k \in \mathbb{N}$,
\begin{align}\label{SDP-SOCP} \sup_{\substack{\mu \in \mathbb{R}, c_{\bf p,q} \ge 0, \\
\sigma_1\in \Sigma^2_{n_1,r}, \sigma_2\in {\bf SDSOS}_{n_2,r}}}\left\{\mu\mid f-\sum_{{\bf p,q} \in {(\mathbb{N}_0)^m}, |{\bf p}|+|{\bf q}| \le k}c_{{\bf p, q}}\prod_{i=1}^{m} \widehat{g}_i^{p_i}(1-\widehat{g}_i)^{q_i}-\mu =\sigma_1+\sigma_2\right\}.\tag{\rm RP$_{k}^{r}$} \end{align}

{\bf Bounded Degree Mixed SDP-SOCP reformulation for the relaxation problem:} We now show that the relaxation problem~\eqref{SDP-SOCP} can be reformulated as mixed semidefinite and second-order cone
programming problems where the number and the size of the mixed SDP-SOCP constraints are independent of the level $k$ of the approximation. Define
\[
h_{{\bf p,q}}(x)= \prod_{i=1}^{m} \widehat{g}_i^{p_i}(1-\widehat{g}_i)^{q_i}, \ \ {\bf p,q} \in {(\mathbb{N}_0)^m}, |{\bf p}|+|{\bf q}| \le k.
\]
Let $\big( f-\sum\limits_{{\bf p,q} \in {(\mathbb{N}_0)^m}, |{\bf p}|+|{\bf q}| \le k}c_{{\bf p,q}} (h_{\bf p,q})\big)(x)=\sum\limits_{\alpha}\big( f-\sum\limits_{{\bf p,q} \in {(\mathbb{N}_0)^m}, |{\bf p}|+|{\bf q}| \le k}c_{{\bf p,q}} (h_{\bf p,q})\big)_{\alpha}x^{i(\alpha)}$ for  each $x\in\mathbb{R}^n$.
Then, \eqref{SDP-SOCP} can be equivalently reformulated as the following mixed semidefinite and second-order cone programming problem:
\begin{eqnarray*}
\sup_{\substack{\mu \in \mathbb{R}, c_{\bf p, q} \ge 0, \\ Q^1 \in S^{s(\frac{r}{2},n_1)},\\ Q^2, M^{ij} \in S^{s(\frac{r}{2},n_2)}}} \{\mu&\mid& \big( f-\sum_{{\bf p,q} \in {(\mathbb{N}_0)^m}, |{\bf p}|+|{\bf q}| \le k}c_{{\bf p,q}} (h_{\bf p,q})\big)_{\alpha}-\mu=\sum_{\substack{ 1\le \beta \le s(\frac{r}{2},n_1), \\ 1\le \gamma \le s(\frac{r}{2},n_1) \\
i(\alpha)=i_+(\beta)+i_+(\gamma)}}Q^1_{\beta \gamma}+\sum_{\substack{1\le \beta \le s(\frac{r}{2},n_2), \\ 1\le \gamma \le s(\frac{r}{2},n_2), \\ i(\alpha)=i^+(\beta)+i^+(\gamma)}}Q^2_{\beta \gamma}, \\ 
 & & Q^1\in S^{ s(\frac{r}{2},n_1)}_+, Q^2=\sum_{1 \le i,j \le s(\frac{r}{2},n_2)} M^{ij}, \\
 & & M^{ij}_{\beta\gamma}=0,\, \forall \ (\beta,\gamma) \notin \{(i,i),(j,j),(i,j),(j,i)\}, \\
& &
 \|\bigg(\begin{array}{c}    2 M^{ij}_{ij}\\
                                                        M^{ij}_{ii}-M^{ij}_{jj}
                                                                                  \end{array}\bigg)\| \le M^{ij}_{ii}+M^{ij}_{jj},  1 \le i,j \le s(\frac{r}{2},n_2) \}.
\end{eqnarray*}
An important and useful feature is that, in the bounded degree semidefinite and second-order cone hierarchy, the number and the size of the mixed semidefinite and second-order cone constraints are independent of the level $k$ of the approximation problems \eqref{SDP-SOCP} for each fixed $r$. This is in contrast to the usual sums-of-squares hierarchy where the number/size of the associated
    conic constraints sharply increases as the level of approximation increases.

  We now formulate the Lagrangian dual problem of the above mixed semidefinite and second-order cone programming reformulation of problem \eqref{SDP-SOCP}.
  Recall that for any $u \in \mathbb{N}$ and  given $y=(y_\alpha)\in \mathbb{R}^{s(u,n)},$   $L_y:\mathbb{R}_d[x]\to \mathbb{R}$ is the Riesz functional given by
  \begin{align}\label{Ly} L_y(f)=\sum_{\alpha=1}^{s(u,n)}f_\alpha y_\alpha \mbox{ for } f(x)=\sum_{\alpha=1}^{s(u,n)}f_\alpha x^{(d)}_\alpha, \end{align}
and ${\bf M}_u(y)$ is the moment matrix about $x \in \R^n$ with degree $u$ generated by $y=(y_\alpha)\in \mathbb{R}^{s(2u,n)}$ which is  defined by
\begin{align} \label{M-2}{\bf M}_u(y)=\sum_{\alpha=1}^{ s(2u,n)} y_\alpha M_\alpha,\end{align}
where for each $\alpha=1,\ldots, s(2u,n),$ $M_\alpha$ is the $(s(u,n)\times s(u,n))$ symmetric matrix such that \begin{align}\label{M-1}x^{(u)}(x^{(u)})^T= \sum_{\alpha=1}^{ s(2u,n)} x^{(u)}_\alpha M_\alpha.\end{align}
It is known that (cf. \cite[Lemma 4.2]{Laurent}) if there is a representing measure $\mu$ of $y$ over $K$ up to order $2u$ in the sense that $y_{\alpha}=\int_K x^{i(\alpha)} d\mu$ for all $\alpha=1,\ldots,s(2u,d)$, then ${\bf M}_{u}(y) \succeq 0$.
Let $X \subseteq \{x_1,\ldots,x_n\}$. We also define ${\bf M}_u^{X}(y)$ to be the moment submatrix
obtained from ${\bf M}_u(y)$ by retaining only those rows and columns $\alpha$ such that ${\rm supp}(i(\alpha)) \in X$ where ${\rm supp}(v)$ is the support of a vector $v \in \mathbb{R}^n$  given by ${\rm supp}(v)=\{i: v_i \neq 0\}$
.

Let $D= \max\{kd, r\}$ and denote $X_1:=(x_1,\ldots,x_{n_1})$, $X_2:=(x_{n_1+1},\ldots,x_n)$. Then, the Lagrangian dual problem of the above mixed semidefinite and second-order cone programming reformulation of problem \eqref{SDP-SOCP}  can be written as
  \begin{eqnarray}\label{SDPSOCP*}
(DRP_k^r) & \displaystyle \inf_{\substack{y=(y_\alpha)\in \mathbb{R}^{s(D,n)}}} & L_y(f) \nonumber \\
& & L_y (h_{\bf p,q}) \ge 0, \, {\bf p,q} \in {(\mathbb{N}_0)^m}, |{\bf p}|+|{\bf q}| \le k, \nonumber \\
& & y_1=1,\nonumber \\
 && {\bf M}_{\frac{r}{2}}^{X_1}(y) \  
 \in S_+^{s(\frac{r}{2},n_1)}, \nonumber \\ 
& & \|\bigg(\begin{array}{c}    2({\bf M}_{\frac{r}{2}}^{X_2}(y))_{ij}\\
                                                      ({\bf M}_{\frac{r}{2}}^{X_2}(y))_{ii}-({\bf M}_{\frac{r}{2}}^{X_2}(y))_{jj}
                                                                                  \end{array}\bigg)\| \le ({\bf M}_{\frac{r}{2}}^{X_2}(y))_{ii}+({\bf M}_{\frac{r}{2}}^{X_2}(y))_{jj}, 1 \le i,j \le s(\frac{r}{2},n_2). \nonumber
\end{eqnarray}


We now establish the asymptotic convergence of the bounded degree SDP-SOCP hierarchy. To achieve this,
we will need the following certificate of positivity of a polynomial over a compact semi-algebraic set.
\begin{lemma} {\bf (Krivine-Stengle's certificate of positivity, cf. \cite[Theorem 2.23]{Lassere_book})}
Let $f$ and $g_i$, $i=1,\ldots,m$ be real polynomials and let $K=\{x:g_i(x) \ge 0, i=1,\ldots,m\}$. Suppose that Assumption A holds. Let $M$ be a positive number such that
$\displaystyle M > \max_{1 \le i\le m}\sup_{x \in K}\{g_i(x)\}$ and denote $\widehat{g}_i(x):=g_i(x)/M$. If $f(x)>0$ for all $x \in K$, then,
there exist nonnegative scalars $c_{{\bf \alpha, \beta}}$ such that
\[
f= \sum_{{\bf \alpha,\beta} \in {(\mathbb{N}_0)^m}}c_{{\bf \alpha, \beta}}\prod_{i=1}^{m} \widehat{g}_i^{\alpha_i}(1-\widehat{g}_i)^{\beta_i}.
\]
\end{lemma}

We now establish the asymptotic convergence of the SDP-SOCP hierarchy as well as the duality between the relaxation problems of the hierarchy and its dual, by making use of the Krivine-Stengle's certificate of positivity.

\begin{theorem}{\bf (Convergence of bounded degree SDP-SOCP hierarchy and duality)} \label{thm:1}
For problem~\eqref{P}, let $M$ be a positive number such that $\displaystyle M > \max_{1 \le i\le m}\sup_{x \in K}\{g_i(x)\}$ and denote $\widehat{g}_i(x)=\frac{g_i(x)}{M}$.  Let $n_1, n_2\in\mathbb{N}$,  $0\le n_1, n_2\le n, n=n_1+n_2$ and fix a positive
even number ${r} \in \mathbb{N}$.
Suppose that Assumption A holds. Then,
\begin{itemize}
\item[{\rm (i)}] ${\rm val}\eqref{SDP-SOCP} \le {\rm val} ({\rm RP}_{k+1}^r) \le {\rm val} ({\rm DRP}_{k+1}^r) \le {\rm val} ({\rm DRP}_{k}^r) \le {\rm val}\eqref{P}$ for all $k\in\mathbb{N}$, and
\[
\lim_{k \rightarrow \infty} {\rm val}\eqref{SDP-SOCP} =\lim_{k \rightarrow \infty} {\rm val}({\rm DRP}_k^r)={\rm val}~\eqref{P}.
\]
\item[{\rm (ii)}] If the feasible set $K$ of (P) has a nonempty interior and there exists $x_0\in K$
such that ${g}_i(x_0)>0$ for all $i = 1,\ldots,m$, then,  for each $k \in \mathbb{N}$
\[
{\rm val}\eqref{SDP-SOCP}={\rm val}({\rm DRP}_k^r)
\]
and the optimal value of ${\rm val}({\rm RP}_k^r)$ is attained.
\end{itemize}
\end{theorem}
\begin{proof} \ [Proof of {\rm (i)}] Firstly, from the construction of the problem~\eqref{SDP-SOCP}, ${\rm val}~\eqref{SDP-SOCP} \le {\rm val} ({\rm RP}_{k+1}^r)$ for all $k$. In addition, take any feasible point $x$ of problem~\eqref{P} and
any feasible point $\mu \in \mathbb{R}, c_{\bf p,q} \ge 0,
\sigma_1\in \Sigma^2_{n_1,r}, \sigma_2\in {\bf SDSOS}_{n_2,r}$ for \eqref{SDP-SOCP}.  By the construction of  \eqref{SDP-SOCP},  we have $\widehat{g}_i(x) \ge 0$ and $1-\widehat{g}_i(x) \ge 0$. Then,
 \[
  f(x)=\sum_{{\bf p,q} \in {(\mathbb{N}_0)^m}, |{\bf p}|+|{\bf q}| \le k}c_{{\bf p, q}}\prod_{i=1}^{m} \widehat{g}_i(x)^{p_i}(1-\widehat{g}_i(x))^{q_i}+\mu+\sigma_1(x)+\sigma_2(x) \ge \mu.
 \]
So, ${\rm val}~\eqref{SDP-SOCP} \le {\rm val} ({\rm RP}_{k+1}^r) \le {\rm val}\eqref{P}$ for all $k$. In particular, $\lim_{k \rightarrow \infty} {\rm val}\eqref{SDP-SOCP}$ exists and
\[
\lim_{k \rightarrow \infty} {\rm val}\eqref{SDP-SOCP}\le {\rm val}\eqref{P}.
\]

To see the reverse inequality, let $\epsilon>0$ and denote
\[
v_k:=  \sup_{\mu \in \mathbb{R}, c_{\bf p,q} \ge 0}\left\{\mu: f-\sum_{{\bf p,q} \in {(\mathbb{N}_0)^m}, |{\bf p}|+|{\bf q}| \le k}c_{{\bf p, q}}\prod_{i=1}^{m} \widehat{g}_i^{p_i}(1-\widehat{g}_i)^{q_i}-\mu \equiv 0\right\}.
\]
It is clear that $v_k \le {\rm val}\eqref{SDP-SOCP}$. Moreover, as $f-{\rm val}\eqref{P}+\epsilon >0$ over $K=\{g_i(x) \ge 0, i=1,\ldots,m\}$, the Krivine-Stengle’s certificate implies that there exists $k_0$ such that
\[
f-\sum_{{\bf p,q} \in {(\mathbb{N}_0)^m}, |{\bf p}|+|{\bf q}| \le {k_0}}c_{{\bf p, q}}\prod_{i=1}^{m} \widehat{g}_i^{p_i}(1-\widehat{g}_i)^{q_i}-({\rm val}\eqref{P}-\epsilon) \equiv 0
\]
and so, $v_{k_0} \ge {\rm val}\eqref{P}-\epsilon$. This shows that $\limsup_{k \rightarrow \infty}v_k \ge {\rm val}\eqref{P}$. So,
\[
\lim_{k \rightarrow \infty} {\rm val} \eqref{SDP-SOCP} \ge  \limsup_{k \rightarrow \infty}v_k \ge {\rm val}\eqref{P}.
\]
Thus, $\lim_{k \rightarrow \infty} {\rm val} \eqref{SDP-SOCP}={\rm val}\eqref{P}$.

Now, note from the construction of $(DRP_k^r)$ and standard weak duality of convex conic problems, one has ${\rm val} ({\rm DRP}_k^r) \ge {\rm val} \eqref{SDP-SOCP}$ and ${\rm val} ({\rm DRP}_{k+1}^r) \le {\rm val} ({\rm DRP}_{k}^r)$. We now see that ${\rm val} ({\rm DRP}_k^r) \le {\rm val}(P)$.
To see this, let $\overline{x}$ be a solution for (P) (this is possible as the feasible set of (P) is compact). Let $\overline{y}=(\overline{y}_{\alpha})$ with $\overline{y}_{\alpha}=\overline{x}^{i(\alpha)}$. As $\overline{x}$ is feasible for (P), $h_{\bf p,q}(\bar x) \ge 0$, for all ${\bf p,q} \in {(\mathbb{N}_0)^m}, |{\bf p}|+|{\bf q}| \le k$, and so, $L_{\overline{y}}(h_{\bf p,q})=\sum_{\alpha}(h_{\bf p,q})_{\alpha} \overline{y}_{\alpha}=\sum_{\alpha}(h_{\bf p,q})_{\alpha} \overline{x}^{i(\alpha)} = h_{\bf p,q}(\overline{x}) \ge 0$. Moreover,  direct verification shows that
${\bf M}_{\frac{r}{2}}(\overline{y}) \succeq 0$ (this can also be seen by the fact that the Dirac measure for $\overline{x}$ is a representing measure of $\overline{y}$). Thus, $\overline{y}$ is feasible for $({\rm DRP}_k^r)$ and so, ${\rm val} ({\rm DRP}_k^r) \le L_{\overline{y}}(f)=f(\overline{x})={\rm val}(P)$. Therefore, we see that ${\rm val}\eqref{SDP-SOCP} \le {\rm val} ({\rm RP}_{k+1}^r) \le {\rm val} ({\rm DRP}_{k+1}^r) \le {\rm val} ({\rm DRP}_{k}^r) \le {\rm val}\eqref{P}$. Note that $\lim_{k \rightarrow \infty} {\rm val} \eqref{SDP-SOCP}={\rm val}\eqref{P}$. It follows that
\[
\lim_{k \rightarrow \infty} {\rm val}\eqref{SDP-SOCP} =\lim_{k \rightarrow \infty} {\rm val}({\rm DRP}_k^r)={\rm val}~\eqref{P}.
\]
So, {\rm (i)} follows.

[Proof of {\rm (ii)}]
Let $y$ be the sequence of
moments of the Lebesgue measure $\mu$ on $K$, scaled to be a probability measure, so that $y_1=L_y(1) = 1$.
From the construction, ${\bf M}_{\frac{r}{2}}(y) \succeq 0$. Indeed, we see that $M_{\frac{r}{2}}(y) \succ 0$. Otherwise, there exists $q \in \mathbb{R}^{s(\frac{r}{2},n)}$ with $q \neq 0$ such that $q^T{\bf M}_{\frac{r}{2}}(y) q =0$. Let $Q$ be a polynomial given by $Q(x)=\sum_\alpha q_{\alpha}x^{i(\alpha)}$. This implies that
\[
q^T{\bf M}_{\frac{r}{2}}(y) q= L_{y}(Q^2)=\int_K Q^2(x) \, d \mu=0.
\]
So, $Q= 0$ almost everywhere on $K$, and hence, $Q \equiv 0$ as $Q$ is a polynomial and $K$ has nonempty interior. This implies that $q=0$ which makes contradiction. So,
$M_{\frac{r}{2}}(y) \succ 0$. This implies that
\[
M_{\frac{r}{2}}^{X_1}(y) \succ 0 \mbox{ and } M_{\frac{r}{2}}^{X_2}(y) \succ 0.
\]
%
In particular,
the latter relation implies that
\[
\|\bigg(\begin{array}{c}    2({\bf M}_{\frac{r}{2}}^{X_2}(y))_{ij}\\
                                                      ({\bf M}_{\frac{r}{2}}^{X_2}(y))_{ii}-({\bf M}_{\frac{r}{2}}^{X_2}(y))_{jj}
                                                                                  \end{array}\bigg)\| < ({\bf M}_{\frac{r}{2}}^{X_2}(y))_{ii}+({\bf M}_{\frac{r}{2}}^{X_2}(y))_{jj}, 1 \le i,j \le s(\frac{r}{2},n_2).
\]
 From our assumption, there exists $x_0\in K$
such that ${g}_i(x_0)>0$ for all $i = 1,\ldots,m$. Recall that $\widehat{g}_i(x)=\frac{g_i(x)}{M}$ and $\displaystyle M > \max_{1 \le i\le m}\sup_{x \in K}\{g_i(x)\}$. So,
\[
0< \widehat{g}_i(x_0)<1 \mbox{ for all } i = 1,\ldots,m.\]
Note that
\[
h_{{\bf p,q}}(x)= \prod_{i=1}^{m} \widehat{g}_i^{p_i}(1-\widehat{g}_i)^{q_i}, \ \ {\bf p,q} \in {(\mathbb{N}_0)^m}, |{\bf p}|+|{\bf q}| \le k.
\]
Thus, there exists an open set $U \subseteq K$ such that for all $x \in U$
\[
h_{\bf p,q}(x) >0 \mbox{ for all } {\bf p,q} \in {(\mathbb{N}_0)^m}, |{\bf p}|+|{\bf q}| \le k.
\]
So,
\[
L_y(h_{\bf p,q})=\int_K h_{\bf p,q}(x) \, d \mu \ge \int_U h_{\bf p,q}(x) \, d \mu >0.
\]
Thus, $y$ is a strictly feasible solution of $(DRP_k^r)$. Then, the conclusion of {\rm (ii)} follows from standard strong duality result of convex conic problems.
\end{proof}

\medskip
In the case of $n_1=n$ and $n_2=0$, the SDP-SOCP relaxation problem~\eqref{SDP-SOCP} reduces to the following semidefinite programming  (SDP) problems which is known as the bounded degree Lasserre SDP hierarchy  \cite{Lag,Bounded_SOS} \begin{equation}\label{SDP}
({\rm SDP}_{k}^{r}) \ \ \  \sup_{\substack{\mu \in \mathbb{R}, c_{\bf p,q} \ge 0, \\
\sigma \in  \Sigma^2_{n,r}}}\left\{\mu\mid f-\sum_{{\bf p,q} \in {(\mathbb{N}_0)^m}, |{\bf p}|+|{\bf q}| \le k}c_{{\bf p, q}}\prod_{i=1}^{m} \widehat{g}_i^{p_i}(1-\widehat{g}_i)^{q_i}-\mu =\sigma\right\}.  \end{equation}
Its corresponding Lagrangian dual can be written as
  \begin{eqnarray}
({\rm DSDP}_k^r) & \displaystyle \inf_{\substack{y=(y_\alpha)\in \mathbb{R}^{s(D,n)}}} & L_y(f) \nonumber \\
& & L_y (h_{\bf p,q}) \ge 0, \, {\bf p,q} \in {(\mathbb{N}_0)^m}, |{\bf p}|+|{\bf q}| \le k, \nonumber \\
& & y_1=1,\nonumber \\
 && {\bf M}_{\frac{r}{2}}(y) \succeq 0,  \  
\end{eqnarray}
where $D= \max\{kd, r\}$. So, we obtain an asymptotic convergence of the bounded degree SDP hierarchy for the problem~\eqref{P} (see \cite{Bounded_SOS}) as follows.

\begin{corollary}{\bf (Convergence of bounded degree  Lasserre SDP hierarchy and duality)} \label{thm:1-SDP}
For problem~\eqref{P}, let $M$ be a positive number such that {{$\displaystyle M >\max_{1 \le i\le m}\sup_{x \in K}\{g_i(x)\}$}} and denote $\widehat{g}_i(x)=\frac{g_i(x)}{M}$.  Fix a positive
even number ${r} \in \mathbb{N}$.
Suppose that Assumption A holds. Then,
\begin{itemize}
\item[{\rm (i)}] ${\rm val}({\rm SDP}_k^r) \le {\rm val} ({\rm SDP}_{k+1}^r) \le  {\rm val} ({\rm DSDP}_{k+1}^r) \le {\rm val} ({\rm DSDP}_{k}^r) \le {\rm val}\eqref{P}$ for all $k\in\mathbb{N}$, and
\[
\lim_{k \rightarrow \infty} {\rm val}({\rm SDP}_k^r) =\lim_{k \rightarrow \infty} {\rm val}({\rm DSDP}_k^r)={\rm val}\eqref{P}.
\]
\item[{\rm (ii)}] If the feasible set $K$ of (P) has a nonempty interior and there exists $x_0\in K$
such that ${g}_i(x_0)>0$ for all $i = 1,\ldots,m$, then,  for each $k \in \mathbb{N}$
\[
{\rm val}({\rm SDP}_k^r)={\rm val}({\rm DSDP}_k^r)
\]
and the optimal value of ${\rm val}({\rm SDP}_k^r)$ is attained.
\end{itemize}


\end{corollary}
\begin{proof} In this setting,  the problem~\eqref{SDP} is equivalently reformulated as the following semidefinite programming problem:
\begin{eqnarray*}
\sup_{\substack{\mu \in \mathbb{R}, c_{\bf p, q} \ge 0, \\ Q \in S^{s(\frac{r}{2},n)}}} \{\mu&\mid& f_{\alpha}-\sum_{{\bf p,q} \in {(\mathbb{N}_0)^m}, |{\bf p}|+|{\bf q}| \le k}c_{{\bf p,q}} (h_{\bf p,q})_\alpha-\mu =
\sum_{\substack{1 \le \beta,\gamma \le s(\frac{r}{2},n),\\ i(\beta)+i(\gamma)=i(\alpha)}} Q_{\beta \gamma},\\
& & Q\in S^{ s(\frac{r}{2},n)}_+ \}.
\end{eqnarray*} So, the proof follows from Theorem~\ref{thm:1} by considering $n_1=n$ and $n_2=0.$
\end{proof}

\medskip
In the case of $n_1=0$ and $n_2=n$, the SDP-SOCP relaxation problem~\eqref{SDP-SOCP} reduces to the following second-order cone (SOCP) program \begin{equation}\label{SOCP} ({\rm SOCP}_{k}^{r}) \ \ \ \sup_{\substack{\mu \in \mathbb{R}, c_{\bf p,q} \ge 0, \\
\sigma \in {\bf SDSOS}_{n,r}}}\left\{\mu\mid f-\sum_{{\bf p,q} \in {(\mathbb{N}_0)^m}, |{\bf p}|+|{\bf q}| \le k}c_{{\bf p, q}}\prod_{i=1}^{m} \widehat{g}_i^{p_i}(1-\widehat{g}_i)^{q_i}-\mu =\sigma\right\}.  \end{equation}
Its corresponding Lagrangian dual can be written as
  \begin{eqnarray}
({\rm DSOCP}_k^r) & \displaystyle \inf_{\substack{y=(y_\alpha)\in \mathbb{R}^{s(D,n)}}} & L_y(f) \nonumber \\
& & L_y (h_{\bf p,q}) \ge 0, \, {\bf p,q} \in {(\mathbb{N}_0)^m}, |{\bf p}|+|{\bf q}| \le k, \nonumber \\
& & y_1=1,\nonumber \\
& & \|\bigg(\begin{array}{c}    2({\bf M}_{\frac{r}{2}}(y))_{ij}\\
                                                      ({\bf M}_{\frac{r}{2}}(y))_{ii}-({\bf M}_{\frac{r}{2}}(y))_{jj}
                                                                                 \end{array}\bigg)\| \le ({\bf M}_{\frac{r}{2}}(y))_{ii}+({\bf M}_{\frac{r}{2}}(y))_{jj}, 1 \le i,j \le s(\frac{r}{2},n). \nonumber
\end{eqnarray}
where $D= \max\{kd, r\}$.
So, we obtain an asymptotic convergence of the bounded degree SOCP hierarchy for the problem~\eqref{P} as follows.

\begin{corollary}{\bf (Convergence of bounded degree SOCP hierarchy and duality)} \label{thm:1-SOCP}
For problem~\eqref{P}, let $M$ be a positive number such that {{$\displaystyle M > \max_{1 \le i\le m}\sup_{x \in K}\{g_i(x)\}$}} and denote $\widehat{g}_i(x)=\frac{g_i(x)}{M}$.  Fix a positive
even number ${r} \in \mathbb{N}$.
Suppose that Assumption A holds. Then,
\begin{itemize}
\item[{\rm (i)}] ${\rm val}({\rm SOCP}_{k}^r) \le {\rm val} ({\rm SOCP}_{k+1}^r) \le  {\rm val} ({\rm DSOCP}_{k+1}^r) \le {\rm val} ({\rm DSOCP}_{k}^r) \le {\rm val}\eqref{P}$ for all $k\in\mathbb{N}$, and
\[
\lim_{k \rightarrow \infty} {\rm val}({\rm SOCP}_{k}^r) = \lim_{k \rightarrow \infty}{\rm val} ({\rm DSOCP}_{k}^r)= {\rm val}\eqref{P}.
\]
\item[{\rm (ii)}] If the feasible set $K$ of (P) has a nonempty interior and there exists $x_0\in K$
such that ${g}_i(x_0)>0$ for all $i = 1,\ldots,m$, then,  for each $k \in \mathbb{N}$
\[
{\rm val}({\rm SOCP}_k^r)={\rm val}({\rm DSOCP}_k^r)
\]
and the optimal value of ${\rm val}({\rm SOCP}_k^r)$ is attained.
\end{itemize}

\end{corollary}
\begin{proof} In this setting,  the problem~\eqref{SOCP} is equivalently reformulated as the following  second-order cone programming problem:
\begin{eqnarray*}
\sup_{\substack{\mu \in \mathbb{R}, c_{\bf p, q} \ge 0, \\ Q, M^{ij} \in S^{s(\frac{r}{2},n)}}} \{\mu&\mid& f_{\alpha}-\sum_{{\bf p,q} \in {(\mathbb{N}_0)^m}, |{\bf p}|+|{\bf q}| \le k}c_{{\bf p,q}} (h_{\bf p,q})_\alpha-\mu =
 \sum_{\substack{1 \le \beta,\gamma \le s(\frac{r}{2},n),\\ i(\beta)+i(\gamma)=i(\alpha)}} Q_{\beta \gamma}, \\
& &  Q=\sum_{1 \le i,j \le s(\frac{r}{2},n)} M^{ij}, \\
& & M^{ij}_{\beta\gamma}=0,\, \forall \ (\beta,\gamma) \notin \{(i,i),(j,j),(i,j),(j,i)\}, \\
& &
 \|\bigg(\begin{array}{c}    2 M^{ij}_{ij}\\
                                                        M^{ij}_{ii}-M^{ij}_{jj}
                                                                                  \end{array}\bigg)\| \le M^{ij}_{ii}+M^{ij}_{jj}, 1 \le i,j \le s(\frac{r}{2},n) \}.
\end{eqnarray*} So, the proof follows from Theorem~\ref{thm:1} by considering $n_1=0$ and $n_2=n$.
\end{proof}

Next, we present a numerical example to illustrate the proposed bounded degree second-order cone programming hierarchy.

\begin{example}
{Let $n \in \mathbb{N}$ with $n \ge 5$ and consider the following nonconvex polynomial problem:
\begin{eqnarray*}
(EP1) & \displaystyle \min_{x \in \mathbb{R}^n} & \sum_{i=1}^nx_i^{4}- n x_1x_2x_3x_4 \\
& \mbox{ s.t. } & 1-\sum_{i=1}^nx_i^{2} \ge 0, \\
& & 0 \le x_i \le 1, i=1,\ldots,n.
\end{eqnarray*}
By solving the KKT condition, it can be directly verified that the  optimal value is $\frac{4-n}{16}$ and is attained at $x^*=(\frac{1}{2},\frac{1}{2},\frac{1}{2},\frac{1}{2},\underbrace{0,\ldots,0}_{n-4})^T$.

Clearly, Assumption A holds for (EP1). Let $r=4$. For $n=20,30,40$, we solve the $kth$ $(k=1,2)$ approximation problems within the bounded degree second-order cone programming hierarchy and the one within the bounded degree  SOS hierarchy using the polynomial optimization toolbox
SPOT \cite{spot} and the conic
 program solver MOSEK. More explicitly, we first use the polynomial optimization toolbox
SPOT to convert the corresponding relaxation problem within the bounded degree second-order cone programming hierarchy (resp. bounded degree
SOS hierarchy) into a second-order cone program (resp. semi-definite program), then we solve the resulting conic programming problem via the conic
 program solver MOSEK.
We also compare the computation results with the standard SOS hierarchy implemented by the software {Gloptipoly 3} \cite{Gloptpoly3}. The numerical tests are conducted
on a computer with a $2.8$GHz Intel Core $i7$ and $8$GB RAM, equipped with {MATLAB} R2015b.

{\rm  \begin{center}
\begin{tabular}{|l|l|l|l|}
                           \hline
                           $(n,k)$ & Bounded degree  & Bounded degree  & {Gloptipoly 3}   \\
                           & 2nd-order cone hierarchy & SOS hierarchy & \\ \hline
                           $(20,1)$ & output value=$-\infty$ & output value=$-\infty$ &  output value=$-\infty$ \\  \hline
                           $(20,2)$ & output value= $-1.0000$ & output value=$-1.0000$ &output value=$-1.0000$  \\
                            & time=$9.1787$ & time=$277.1627$  & time=$5998.2367$ \\
                            & optimal value: Yes & optimal value: Yes & optimal value: Yes \\
                                                       \hline
                                              $(30,1)$ & output value=$-\infty$ & output value=$-\infty$  & out of memory \\ \hline
                           $(30,2)$ & output value= $-1.6250$ & out of memory  & out of memory \\
                            & time=$40.3081$ &    &  \\
                                                                              & optimal value: Yes &  & \\ \hline
                                                     $(40,1)$ & output value=$-\infty$ & output value=$-\infty$  & out of memory \\ \hline
                           $(40,2)$ & output value= $-2.250$ & out of memory  & out of memory \\
                            & time=$136.1732$ &    & \\
                            & optimal value: Yes &  & \\ \hline
                         \end{tabular}\\
Table 1
\end{center}}
Table 1 summarizes the computed optimal value as well as the CPU time used (measured in seconds). As one can see
 from the table, the 2nd approximation problem in the bounded degree second-order cone programming hierarchy reaches the true optimal value
 of the underlying polynomial optimization problem for all the cases. On the other hand, the bounded degree SOS hierarchy and the SOS hierarchy
 implemented in {Gloptipoly 3} runs out of memory for the cases $n=30,40$.

}
\end{example}
As shown in \cite{Bounded_SOS}, unlike the linear programming hierarchy which does not exhibits finite convergence even
for simple convex programming problems, the bounded degree SOS hierarchy can be exact for a  class of convex polynomial program called SOS-convex polynomial program covering convex quadratic programming problem.
We will see in the next Section that similar conclusions hold for the bounded degree SOCP hierarchy. Indeed, we will show that the bounded degree SOCP hierarchy is exact for a class of nonconvex polynomial program
with suitable sign structure and a class of new convex polynomial program called SOCP-convex polynomial program.

\section{One-Step SOCP Hierarchy Convergence }
In this section,  we   establish one-step convergence of the
bounded degree  SOCP hierarchy   for two classes of polynomial optimization problems. We achieve these results using new  representations of nonnegative polynomials in terms of SDSOS polynomials.

\subsection{SOCP Convexity}

Recall that a differentiable function $f$ on $\mathbb{R}^n$ is convex if and only if
$h_f(x,y):=f(x)-f(y)- \nabla f(y)^T(x-y) \ge 0$ for each $x,y \in \mathbb{R}^n$.

\begin{definition}{\bf (SOCP-Convex Polynomial)}\cite{jeya-li-arxiv}
{\rm Let $f$ be a polynomial on $\mathbb{R}^n$ with degree $d$, where $d \in {\mathbb{N}_0}$. 
Let $h_f$ be a polynomial defined by
$h_f(x,y):= f(x)-f(y)- \nabla f(y)^T(x-y).$
 We say that $f$ is {\it SOCP-convex} whenever $h_f$ is an SDSOS polynomial in the variable of $(x,y)$.}
\end{definition}

Note from the definition that any SOCP-convex polynomial is convex, and for any SOCP-convex polynomials $f,g$ and $\lambda \ge 0$, $f+g$ and $\lambda f$ are also SOCP-convex. An important and interesting feature for SOCP-convexity is that checking whether a given polynomial is SOCP-convex or not  can be done by solving  a  second-order cone programming problem.
Below, we provide simple examples for SOCP-convex polynomials.

\begin{example}{\bf (Examples of SOCP-convex polynomials)}
The following functions are SOCP-convex polynomials
\begin{itemize}
\item[{\rm (i)}] Any separable convex quadratic function with the form that $f=\sum_{i=1}^n v_{i} x_{i}^{2}+ w_i x_i +\alpha_i$, $v_i,w_i,\alpha_i \in \mathbb{R}$ is an
SOCP convex polynomial. To see this, we note that $f(x)-f(y)- \nabla f(y)^T(x-y)= \sum_{i=1}^n [v_i x_i^2-v_iy_i^2 -2v_iy_i(x_i-y_i)]=\sum_{i=1}^nv_i (x_i-y_i)^2$ is an SDSOS polynomial.

 \item[{\rm (ii)}] Any polynomial $f(x)=\sum_{i=1}^n  \gamma_{i} x_{i}^{2d}$ with $\gamma_{i} \ge 0$,  is SOCP-convex. Indeed,
 for any $x,y \in \mathbb{R}^n$ \[
h_f(x,y)= \sum_{i=1}^n \gamma_i (x_i^{2d}+(2d-1)y_i^{2d}-2d y_i^{2d-1}x_i) \ge 0.
\]
This implies that $h_i(x_i,y_i):=\gamma_i( x_i^{2d}+(2d-1)y_i^{2d}-2d y_i^{2d-1}x_i)$ is a nonnegative polynomial on $\mathbb{R}^2$. Direct verification shows that $h_i$ is SDSOS (this
can also be seen from Proposition \ref{th:2} later). Thus, $h_f$ is also an SDSOS polynomial. So, $f=\sum_{i=1}^n  \gamma_{i} x_{i}^{2d}$ with $\gamma_{i} \ge 0$,  is SOCP-convex.

 \item[{\rm (iii)}] An SOCP-convex polynomial with degree greater than $2$ need not to be separable. For example, consider the convex polynomial $f(x_1,x_2)=x_1^4+x_2^4-2x_1^2x_2^2$. It follows that
 for all $(x,y) \in \mathbb{R}^2 \times \mathbb{R}^2$, $$h_f(x,y)=x_1^4+x_2^4-2x_1^2x_2^2-4x_1y_1^3-4x_2y_2^3-6y_1^2y_2^2+3y_1^4+3y_2^4 \ge 0.$$
 Direct calculations show that $h_f$ is SDSOS (this
can also be seen from Proposition \ref{th:2} later). Thus, $f$ is SOCP-convex.
\end{itemize}
\end{example}

\begin{proposition}{\bf (Representation of nonnegative SOCP-convex polynomials)}\label{prop:SOCP-convex}
Let $f$ be an SOCP-convex polynomial. 
Then  $f(x) \ge 0$ for all $x \in \mathbb{R}^n$ if and only if  $f$ is SDSOS.

\end{proposition}
\begin{proof} Denote the degree of the polynomial $f$ as $d$. Clearly, any  SDSOS polynomial is nonnegative. To see the reverse implication,
let $f$ be a nonnegative SOCP-convex polynomial. Then, $d$ is an even number. Moreover, $f$
is a convex polynomial which is bounded below by zero, and so, $f$ attains its global minimum at some point $\bar x$. This shows that $\nabla f(\bar x)=0$ and $f(\bar x) \ge 0$.
Moreover, let $h_f(z)=f(x)-f(y)-\nabla f(y)^T(x-y) \mbox{ for all } z=(x,y).$ From the definition of  SOCP-convexity of $f$,
$h_f$ is an SDSOS polynomial in the variable of $(x,y)$.
So, there exist
 $\hat{\alpha}_i,\hat{\beta}_{ij}^+,\hat{\beta}_{ij}^-,\hat{\gamma}_{ij}^+,\hat{\gamma}_{ij}^- \ge 0$ such that for all $z=(x,y)$
\begin{eqnarray*}
h_f(z) =\sum_{i=1}^{s(\frac{d}{2}, 2n)} \hat{\alpha}_i m_i^2(z)+\sum_{i,j=1,i \neq j}^{s(\frac{d}{2}, 2n)} (\hat{\beta}_{ij}^+ m_i(z)+\hat{\gamma}_{ij}^+ m_j(z))^2+\sum_{i,j=1,i \neq j}^{s(\frac{d}{2}, 2n)} (\hat{\beta}_{ij}^- m_i(z)-\hat{\gamma}_{ij}^- m_j(z))^2.
\end{eqnarray*}
where $m_i(z), m_j(z)$, $1 \le i, j \le s(\frac{d}{2}, 2n)$, are monomial of $z=(x,y)$ with degree up to $\frac{d}{2}$. In particular, let $z=(x,\bar x)$, it follows that there exist ${\alpha}_i,{\beta}_{ij}^+,{\beta}_{ij}^-,{\gamma}_{ij}^+,{\gamma}_{ij}^- \ge 0$ such that
\begin{eqnarray*}
h_f(x,\bar x) =\sum_{i=1}^{s(\frac{d}{2},n)} \alpha_i m_i^2(x)+\sum_{i,j=1, i \neq j}^{s(\frac{d}{2},n)} (\beta_{ij}^+ m_i(x)+\gamma_{ij}^+ m_j(x))^2+\sum_{i,j=1,i \neq j}^{s(\frac{d}{2},n)} (\beta_{ij}^- m_i(x)-\gamma_{ij}^- m_j(x))^2.
\end{eqnarray*}
where $m_i(x), m_j(x)$, $1 \le i, j \le s(\frac{d}{2},n)$, are monomials of $x$ with degree up to $\frac{d}{2}$.
Note that
\begin{eqnarray*}
f(x)-f(\bar x) =  f(x)-f(\bar x)-\nabla f(\bar x)^T(x-\bar x)
& = & h_f(x,\bar x).
\end{eqnarray*}
This shows that
$f(x)= f(\bar x)+ h_f(x,\bar x)$ is an SDSOS polynomial.
\end{proof}

The following simple example shows that the above SDSOS representation for nonnegative polynomial can fail for a convex quadratic function which is not SOCP-convex.
\begin{example} {\bf (Failure of SDSOS representation  in the absence of SOCP-convexity)}
Consider the convex quadratic function $f(x_1,x_2)=(x_1+x_2-1)^2$. We first see that the nonnegative representation in Proposition \ref{prop:SOCP-convex} fails.
Clearly,  $f$ is nonnegative. We now show that $f$ is not SDSOS.
Suppose on the contrary that there exist $\lambda_i \ge 0$, $i=1,\ldots,14$, such that
\begin{eqnarray*}
(x_1+x_2-1)^2&=& \lambda_0+\lambda_1 x_1^2 + \lambda_2 x_2^2 + (\lambda_3 x_1+\lambda_4 x_2)^2+(\lambda_5 x_1-\lambda_6 x_2)^2 \\
& & + (\lambda_7 x_1+\lambda_8)^2+(\lambda_9 x_1-\lambda_{10} )^2+ (\lambda_{11} x_2+\lambda_{12})^2+(\lambda_{13} x_2-\lambda_{14} )^2.
\end{eqnarray*}
By comparing the coefficients, one sees that
$\lambda_3 \lambda_4 \ge 1, \lambda_3^2 \le 1 \mbox{ and } \lambda_4^2 \le 1$.
This shows that $2 \le 2\lambda_3\lambda_4 \le \lambda_3^2+\lambda_4^2 \le 2$ which entails that $\lambda_3=\lambda_4= \pm 1$. So,
$\lambda_1=\lambda_2=\lambda_5=\lambda_6=0$ and $\lambda_7=\lambda_9=\lambda_{11}=\lambda_{13}=0$ and hence
\[
(x_1+x_2-1)^2=\lambda_0+ (x_1+x_2)^2+\lambda_8^2+\lambda_{10}^2+\lambda_{12}^2+\lambda_{14}^2,
\]
which is impossible. So, $f$ is not SDSOS.

Next, we observe that $f$ is not SOCP-convex. Indeed, using a similar line of method as above, one can show that
$h_f(x,y)= (x_1+x_2-y_1-y_2)^2$
 is not an SDSOS polynomial in the variable of $(x,y)$. Thus, $f$ is not an SOCP-convex polynomial.
\end{example}

We now see that, under an SOCP-convexity assumption,  problem (P) enjoys an exact SOCP relaxation, and so, one-step convergence holds for the
bounded degree SOCP hierarchy.
\begin{theorem}\label{th:9} {\bf (Exact SOCP relaxation under SOCP-convexity)}
Let $f,g_i$ be polynomials with an even degree $d$ and let $l=\frac{d}{2}$.
For problem (P), let $x^* \in {\rm Argmin}(P)$ and $\lambda^*=(\lambda_1^*,\ldots,\lambda_m^*) \in \mathbb{R}^m_+$ be a KKT multiplier associated with $x^*$.
Suppose that $f-\sum_{i=1}^m \lambda_i^* g_i$ is SOCP-convex. Then,
\begin{eqnarray}\label{eq:99}
\min(P) &= & \max_{\substack{\mu\in \mathbb{R}, \lambda_i \ge 0, \\
\sigma\in {\bf SDSOS}_{n,d}}}\{\mu \mid f-\sum_{i=1}^m\lambda_i g_i-\mu =\sigma\} \nonumber \\
& = & \max_{\substack{\mu\in \mathbb{R}, \lambda_i \ge 0,\\ M^{ij} \in S^{s(l,n)}}}\{\mu\mid f_{\alpha}-\sum_{i=1}^m\lambda_i (g_i)_{\alpha}-\mu=\sum_{\substack{1 \le \beta,\gamma \le s(l,n) \\ i(\alpha)=i(\beta)+i(\gamma)}} \sum_{1 \le i,j \le s(l,n)} M^{ij}_{\beta \gamma}, \nonumber \\
& & \ \ \ \ \ \ \ \ \ \ \ \ \ \    M^{ij}_{\beta \gamma}=0,\, \forall \ (\beta,\gamma) \notin \{(i,i),(j,j),(i,j),(j,i)\}, \nonumber  \\
& & \ \ \ \ \ \ \ \ \ \ \ \ \ \  \|\bigg(\begin{array}{c}    2 M^{ij}_{ij}\\
                                                        M^{ij}_{ii}-M^{ij}_{jj}
                                                                                  \end{array}\bigg)\| \le M^{ij}_{ii}+M^{ij}_{jj}, 1 \le i,j \le s(l,n)\}.
                                                                                  \end{eqnarray} In particular, one step convergence holds for the bounded degree SOCP hierarchy, that is, ${\rm val}\eqref{P}={\rm val}({\rm SOCP_1^{d}})$, whenever Assumption A holds.
\end{theorem}
\begin{proof}
We first observe that
$\min(P)  \ge   \sup_{\mu\in \mathbb{R}, \lambda_i \ge 0}\{\mu: f-\sum_{i=1}^m\lambda_i g_i-\mu \ge 0\}$
always holds. To see the reverse inequality, define
$h(x):=f(x)-\sum_{i=1}^m\lambda_i^* g_i(x)-f(x^*)$.
Then, $h$ is an SOCP-convex polynomial and so, it is convex.  Since $\lambda^*=(\lambda_1^*,\ldots,\lambda_m^*) \in \mathbb{R}^m_+$ is  a KKT multiplier associated with $x^*$, we have $$\nabla h(x^*)=\nabla f(x^*)-\sum_{i=1}^m\lambda_i^* \nabla g_i(x^*)=0 \mbox{ and } \lambda_i^* g_i(x^*)=0.$$ This implies that
$h$ attains minimum at $x^*$, and so, $h(x) \ge h(x^*)=0$ for all $x \in \mathbb{R}^n$.
So,  $\min(P)=f(x^*) \le \max_{\mu\in \mathbb{R}, \lambda_i \ge 0}\{\mu: f-\sum_{i=1}^m\lambda_i g_i-\mu \ge 0\}$. Therefore, we have
\begin{align}\label{au}
\min(P)  =  \max_{\mu\in \mathbb{R}, \lambda_i \ge 0}\{\mu: f-\sum_{i=1}^m\lambda_i g_i-\mu \ge 0\},
\end{align}
and the maximum is attained at $\mu=f(x^*)$ and $\lambda_i=\lambda_i^*$, $i=1,\ldots,m$. Note here that it is sufficient to  consider $ \mu\in\mathbb{R}, \lambda_i \ge 0$ in \eqref{au} such that $f-\sum_{i=1}^m\lambda_i g_i-\mu$ is SOCP-convex polynomial with degree $d$.

Now, fix $\lambda_i \ge 0$ and $\mu \in \mathbb{R}$ such that
$f-\sum_{i=1}^m\lambda_i g_i-\mu$ is a nonnegative SOCP-convex polynomial with degree $d$. Then, by Proposition \ref{prop:SOCP-convex},
$f-\sum_{i=1}^m\lambda_i g_i-\mu$ is an  SDSOS polynomial with degree $d$. Thus,
\[
\min(P) =  \max_{\substack{\mu\in \mathbb{R}, \lambda_i \ge 0, \\
\sigma\in  {\bf SDSOS}_{n,d}}}\{\mu\mid f-\sum_{i=1}^m\lambda_i g_i-\mu =\sigma\}.
\]
The second equality of \eqref{eq:99} holds by the SOCP reformulation of an SDSOS polynomial as given in Proposition \ref{prop:1-SDP}.
\end{proof}

\begin{remark}{(\bf SDSOS Representation Condition and One-step Convergence).}
Define an SDSOS representation condition (C) for a polynomial $p$ as follows: for any $\mu \in \mathbb{R}$
\[
(C) \ \ \ \ \ \  p \ge \mu \ \Rightarrow \ p-\mu \mbox{ is SDSOS}.
\]
A close inspection of the proof of Theorem \ref{th:9} shows that the conclusion of Theorem \ref{th:9} continues to hold if
 $h:=f-\sum_{i=1}^m \lambda_i^* g_i$ is a convex polynomial which satisfies the above condition (C),
where $\lambda_i^*$ are KKT multipliers of problem (P).
Moreover, we note that this abstract condition can also be satisfied for a convex polynomial which is not SOCP-convex. For example,
 $f(x_1,x_2)=x_1^6+x_1^4+x_1^2-2x_1x_2+x_2^2$ is a convex polynomial satisfies condition (C) and it is not SOCP-convex.
\end{remark}

\subsection{Polynomial Optimization with Essentially Nonpositive Coefficients}
Let us now present another class of polynomials, called, polynomials with essentially nonpositive coefficients, that admits one-step convergence.

We first recall the definition of
polynomials with essentially non-positive coefficients introduced in \cite{HLQ}.

\begin{definition}
Let $f$ be a polynomial on $\mathbb{R}^n$ with degree $d$. Let $r=f(0)$ be the constant term of $f$ and let $f_{d,i}$ be the coefficient associated with $x_i^{d}$. Then, $f$ can be written as
\[
f(x)=\sum_{i=1}^n f_{d,i} x_i^{d}+\sum_{{\bf p} \in \Omega_f}f_{\bf p}x_1^{p_1}\cdots x_n^{p_n}+r.
\]
where  \begin{equation}\label{eq:Omega_f}\Omega_f=\{{\bf p} \in {(\mathbb{N}_0)^n} : f_{{\bf p}} \neq 0 \mbox{ and } 0<\max_{1 \le i \le n}p_i <d\}. \end{equation}
 We say $f$ has essentially nonpositive coefficients if $f_{\bf p} \le 0$ for all ${\bf p} \in \Omega_f$.
\end{definition}

For a real polynomial $f$,  define a new polynomial $\hat{f}$, given by
\[
\hat{f}(x)=\sum_{i=1}^n f_{d,i} \, x_i^{d}-\sum_{{\bf p} \in \Delta_f}|f_{\bf p}|x_1^{p_1} \cdots x_n^{p_n},
\]
where $\Delta_f:=\{{\bf p}  \in \Omega_f: f_{\bf p} < 0 \mbox{ or } {\bf p} \notin (2\mathbb{N} \cup \{0\})^n\}$.
We now recall the following useful lemma, which provides a test for SDSOS property of a homogeneous polynomial $f$ in terms of the nonnegativity of a new function $\hat{f}$.
\begin{lemma} {\rm (\cite[Corollary 2.8]{Fiaco})} \label{lemma:2.1}
Let $f$ be a homogeneous polynomial of degree
$d$ where $d$ is an even number. If $\hat{f}$ is a nonnegative polynomial, then $f$ is an SDSOS polynomial (or sums-of-binomial-squares polynomial).
\end{lemma}

\begin{proposition}\label{th:2} {\bf (SDSOS representation for nonnegative polynomials with essentially nonpositive coefficients)}
Let $f$ be a polynomial with essentially nonpositive coefficients. If $f$ is a nonnegative polynomial, then $f$ is an SDSOS polynomial.
\end{proposition}
\begin{proof}
We first consider the case where $f$ is a homogeneous polynomial.
Note that $f$ is a nonnegative polynomial with essentially nonpositive coefficients. We first observe that
$\Delta_f=\Omega_f$ in this case. Moreover, for all ${\bf p} \in \Delta_f (=\Omega_f)$, $f_{{\bf p}} \le 0$. Then, one has
\begin{eqnarray*}
\hat{f}(x)= \sum_{i=1}^n f_{d,i} \, x_i^{d}-\sum_{{\bf p} \in \Delta_f}|f_{{\bf p}}|x_1^{p_1} \cdots x_n^{p_n}
&= & \sum_{i=1}^n f_{d,i} \, x_i^{d}+\sum_{{\bf p} \in \Delta_f}f_{{\bf p}}x_1^{p_1} \cdots x_n^{p_n} \\
&= & \sum_{i=1}^n f_{d,i} \, x_i^{d}+\sum_{{\bf p} \in \Omega_f}f_{{\bf p}}x_1^{p_1} \cdots x_n^{p_n} \\
& = & f(x).
\end{eqnarray*}
So, $\hat{f}$ is a nonnegative polynomial, and hence, Lemma \ref{lemma:2.1} implies that $f$ is an SDSOS polynomial.

In the more general case that $f$ is a nonhomogeneous polynomial, we can decompose it as
\[
f(x)=\sum_{i=1}^n f_{d,i} x_i^{d}+\sum_{{\bf p} \in \Omega_f \backslash \{0\}}f_{{\bf p}}\, x_1^{p_1} \cdots x_n^{p_n}+r,
\]
where $r=f(0)$ and $f_{{\bf p}} \le 0$ for all ${\bf p} \in \Omega_f \backslash \{0\}$.
Its canonical homogenization can be written as
\[
\tilde{f}(x,t)=\sum_{i=1}^n f_{d,i} x_i^{d}+\sum_{{\bf p} \in \Omega_f}f_{{\bf p}}\, x_1^{p_1} \cdots x_n^{p_n}t^{m-|{\bf p}|}+r\, t^d \mbox{ for all } (x^T,t)^T \in \mathbb{R}^{n+1}.
\]
As $f$ is nonnegative polynomial with essentially nonpositive coefficients, $\tilde{f}$ is also a nonnegative polynomial with essentially nonpositive coefficients. Then, the homogeneous cases implies that
$\tilde{f}$ is an SDSOS polynomial. Therefore, $f(x)=\tilde{f}(x,1)$ is also an SDSOS polynomial.
\end{proof}

 As a simple consequence of Proposition \ref{th:2}, we now obtain the exact SOCP relaxation result for nonconvex polynomial optimization with essentially nonpositive coefficients,
 and so, one-step convergence of the
bounded degree SOCP hierarchy holds. In the
 quadratic cases (that is,
 $d=2$), this result was obtained in \cite{kim} (for a related result see also \cite{jll}).
  \begin{theorem} {\bf (Exact SOCP relaxation for polynomial programs with essentially nonpositive coefficients)}
   Let $f$ and $-g_i$ $i=1,\ldots,m$, be polynomials on $\mathbb{R}^n$ with essentially nonpositive coefficients and an even degree $d$.
   For problem (P), suppose that the
  strict feasibility condition holds, i.e., there exists $x_0 \in \mathbb{R}^n$ such that $g_i(x_0)>0$ for all $l=1,\ldots,p$.
 Then, we have
\begin{eqnarray*}
{\rm val}(P) &= & \max_{\substack{\mu\in \mathbb{R}, \lambda_i \ge 0 \\
\sigma \in {\bf SDSOS}_{n,d}}}\{\mu \mid f-\sum_{i=1}^m\lambda_i g_i-\mu = \sigma\}.
                                                                                  \end{eqnarray*}
                                                                                  In particular, one step convergence holds for the bounded degree SOCP hierarchy, that is, ${\rm val}(P)={\rm val}(SOCP_1^{d})$, whenever Assumption A holds.
 \end{theorem}
 \begin{proof}
 As $f$ and $-g_i$ $i=1,\ldots,m$, be polynomials on $\mathbb{R}^n$ with essentially nonpositive coefficients and an even degree $d$, and strict feasibility condition holds,
 it has been shown in \cite{HLQ} that exact sums-of-squares relaxation holds, that is,
\begin{eqnarray} \label{eq:lemma_2}
{\rm val}(P) =\max_{\substack{\mu\in \mathbb{R}, \lambda_i \ge 0 \\
\sigma \in \Sigma^2_{n,d}}}\{\mu\mid f-\sum_{i=1}^m\lambda_i g_i-\mu = \sigma\}.
\end{eqnarray}
Let $\lambda_i \ge 0$ and $\mu \in \mathbb{R}$ be such that $f-\sum_{i=1}^m\lambda_i g_i-\mu$ is  sums-of-squares. Then, by our assumption, we see that
   $f-\sum_{i=1}^m\lambda_i g_i-\mu$ is a nonnegative polynomial with essentially nonpositive coefficients. Theorem \ref{th:2} then implies that
   $f-\sum_{i=1}^m\lambda_i g_i-\mu$ is an SDSOS polynomial. So,
   \eqref{eq:lemma_2} gives us that
\begin{eqnarray*}
{\rm val}(P) &= & \max_{\substack{\mu\in \mathbb{R}, \lambda_i \ge 0 \\
\sigma \in {\bf SDSOS}_{n,d}}}\{\mu\mid f-\sum_{i=1}^m\lambda_i g_i-\mu = \sigma\}.
                                                                                  \end{eqnarray*}

Finally, if Assumption A holds, then Theorem \ref{thm:1} implies that ${\rm val} ({\rm SOCP}_k^d) \le {\rm val} ({\rm SOCP}_{k+1}^d) \le {\rm val}(P)$ for all $k$, and
$\lim_{k \rightarrow \infty} {\rm val} ({\rm SOCP}_k^d) ={\rm val}(P)$.
Observe that $$
{\rm val}({\rm SOCP}_1^d) \ge \max_{\substack{\mu\in \mathbb{R}, \lambda_i \ge 0 \\
\sigma \in {\bf SDSOS}_{n,d}}}\{\mu\mid f-\sum_{i=1}^m\lambda_i g_i-\mu = \sigma\}={\rm val}(P).
$$
This implies that ${\rm val}({\rm SOCP}_1^d)={\rm val}(P)$.
 \end{proof}

\begin{example}
{\rm Consider the following nonconvex polynomial optimization problem with degree $10$:
\begin{eqnarray*}
(EP2) & \displaystyle \min_{x_1,x_2,x_3,x_4 \in \mathbb{R}} & \sum_{i=1}^4x_i^{10}-10 x_1x_2x_3x_4 \\
& \mbox{ s.t. } & 1-\sum_{i=1}^4x_i^{10} \ge 0, \\
& & 0 \le x_i \le 1, \ i=1,2,3,4.
\end{eqnarray*}
By solving the KKT condition, it can be directly verified that the global minimum is optimal value $1-10(\sqrt[10]{\frac{1}{4}})^4 \approx -4.7435$ and is attained at $(x_1^*,x_2^*,x_3^*,x_4^*)=(\sqrt[10]{\frac{1}{4}},\sqrt[10]{\frac{1}{4}},\sqrt[10]{\frac{1}{4}},\sqrt[10]{\frac{1}{4}})
$. 

Clearly, Assumption A holds for problem (EP2). Note that the degree of the functions involved in this polynomial
optimization problem is $d=4$. Let $r=4$ and consider the first relaxation problem $(SOCP^4_1)$ in the bounded degree
hierarchy problem $(SOCP^{4}_k)$.  We first convert this problem into a second-order cone programming problem using the polynomial optimization toolbox
SPOT \cite{spot}. Solving the corresponding second-order cone programming problem via the popular conic
 program solver MOSEK, we obtain the optimal value $-4.7435$ (in $43.73$ seconds) which coincides with the true optimal value.

  Moreover,
one can also solve this problem via the recently established bounded degree sums-of-squares hierarchy \cite{Bounded_SOS}. Indeed, using the
  the polynomial optimization toolbox
SPOT \cite{spot} to convert the first relaxation problem in the bounded sums-of-squares hierarchy into a semi-definite programming problem, and then solve it via  MOSEK. We also obtain the true optimal value $-4.7435$ but use more time ($2934.75$
seconds) in accomplishing the task.}
\end{example}

\section{Retrieving Solutions from  Exact SOCP Relaxation}
In this section, we present a version of Jensen's inequality for SOCP-convex polynomials which is then employed to  extract    solutions of problem~\eqref{P} from its SOCP relaxation.

For any $u \in \mathbb{N}$ and  $y=(y_\alpha)\in \mathbb{R}^{s(u,n)},$ recall that $L_y(f)$ is the Riesz functional defined as in \eqref{Ly} and
${\bf M}_{u}(y)$ is the moment matrix with respect to $y=(y_\alpha)\in \mathbb{R}^{s(2u,n)}$ defined as in \eqref{M-2}. We first establish
a Jensen's inequality for SOCP-convex polynomials.

\begin{proposition}{\bf (Jensen's inequality with SOCP-convex polynomials)}\label{prop:Jensen}
Let $f$ be an SOCP-convex polynomial on $\mathbb{R}^n$ with degree $d:=2l$. Let $y=(y_\alpha)\in \mathbb{R}^{s(d,n)}$ with $y_1=1$ and  \begin{align}\label{1.12}  \|\bigg(\begin{array}{c}    2({\bf M}_l(y))_{ij}\\
                                                      ({\bf M}_l(y))_{ii}-({\bf M}_l(y))_{jj}
                                                                                  \end{array}\bigg)\| \le ({\bf M}_l(y))_{ii}+({\bf M}_l(y))_{jj}, \; 1 \le i,j \le s(l,n).\end{align} Then, we have $$L_y(f)\ge f\big(L_y(x_1),\ldots, L_y(x_n)\big),$$  where $L_y$ is given as in \eqref{Ly} and  $x_i$ denotes the polynomial which maps a vector $x$ in $\mathbb{R}^n$ to its $i$th coordinate.

\end{proposition}
\begin{proof} Fix $\bar z\in\mathbb{R}^n$ but arbitrary and consider the following expansion \begin{align}\label{1.11}f(x)-f(\bar z)=\nabla f(\bar z)^T(x-\bar z) +g(x,\bar z),\end{align} where $g(x,\bar z)$ is  a polynomial.  Applying $L_y$ on both sides in \eqref{1.11}, we have \begin{align*}\label{}L_y(f)-f(\bar z)=\nabla f(\bar z)^T\big(L_y(x)-\bar z\big) +L_y\big(g(\cdot,\bar z)\big),\end{align*} where we should recall that $y_1=1.$  Since $f$ is an SOCP-convex polynomial on $\mathbb{R}^n$, $g(\cdot,\bar z)$   is an SDSOS polynomial on $\mathbb{R}^n$, too. Then,  there exist ${\alpha}_i,{\beta}_{ij}^+,{\beta}_{ij}^-,{\gamma}_{ij}^+,{\gamma}_{ij}^- \ge 0$ such that
\begin{eqnarray*}
g(x,\bar z) =\sum_{i=1}^{s(l,n)} \alpha_i m_i^2(x)+\sum_{i,j=1, i \neq j}^{s(l,n)} (\beta_{ij}^+ m_i(x)+\gamma_{ij}^+ m_j(x))^2+\sum_{i,j=1, i \neq j}^{s(l,n)} (\beta_{ij}^- m_i(x)-\gamma_{ij}^- m_j(x))^2.
\end{eqnarray*}
where $m_i(x), m_j(x)$, $1 \le i,j \le s(l,n)$, are monomials of $x$ with degree up to $l$. Hence, \begin{align*}
L_y\big(g(\cdot,\bar z)\big) =&L_y\big(\sum_{i=1}^{s(l,n)} \alpha_i m_i^2(x)\big)+L_y\big(\sum_{i,j=1, i \neq j}^{s(l,n)} (\beta_{ij}^+ m_i(x)+\gamma_{ij}^+ m_j(x))^2\big)\\&+L_y\big(\sum_{i,j=1, i \neq j}^{s(l,n)} (\beta_{ij}^- m_i(x)-\gamma_{ij}^- m_j(x))^2\big)\\&=\sum_{i=1}^{s(l,n)} \alpha_i ({\bf M}_l(y))_{ii}+\sum_{i,j=1, i \neq j}^{s(l,n)} (\beta_{ij}^+)^2 ({\bf M}_l(y))_{ii}+2\sum_{i,j=1, i \neq j}^{s(l,n)} \beta_{ij}^+\gamma_{ij}^+({\bf M}_l(y))_{ij}\\&+\sum_{i,j=1, i \neq j}^{s(l,n)} (\gamma_{ij}^+)^2 ({\bf M}_l(y))_{jj}+\sum_{i,j=1, i \neq j}^{s(l,n)} (\beta_{ij}^-)^2 ({\bf M}_l(y))_{ii}-2\sum_{i,j=1, i \neq j}^{s(l,n)} \beta_{ij}^-\gamma_{ij}^-({\bf M}_l(y))_{ij}\\&+\sum_{i,j=1, i \neq j}^{s(l,n)} (\gamma_{ij}^-)^2 ({\bf M}_l(y))_{jj}\ge 0,
\end{align*} where the inequality holds  by observing by \eqref{1.12} that $ ({\bf M}_l(y))_{ii}\ge 0,  ({\bf M}_l(y))_{jj}\ge 0$ and $ \sqrt{({\bf M}_l(y))_{ii} ({\bf M}_l(y))_{jj}}\ge  ({\bf M}_l(y))_{ij},  1 \le i,j \le s(l,n).$ So, we arrive at \begin{align*}\label{}L_y(f)-f(\bar z)\ge \nabla f(\bar z)^T\big(L_y(x)-\bar z\big).\end{align*} Taking $\bar z:=L_y(x)=\big(L_y(x_1),\ldots, L_y(x_n)\big),$ { we obtain the desired conclusion}.
\end{proof}

\medskip
 The  Lagrangian dual  of the  SOCP relaxation problem  in  \eqref{eq:99} is 
\begin{eqnarray}\label{SOCP*}
& \displaystyle \inf_{\substack{y=(y_\alpha)\in \mathbb{R}^{s(d,n)}}} & L_y(f) \nonumber \\
& & L_y(g_i)\ge 0, i=1,\ldots, m,\nonumber \\
& & y_1=1,\nonumber \\
& & \|\bigg(\begin{array}{c}    2({\bf M}_l(y))_{ij}\\
                                                      ({\bf M}_l(y))_{ii}-({\bf M}_l(y))_{jj}
                                                                                  \end{array}\bigg)\| \le ({\bf M}_l(y))_{ii}+({\bf M}_l(y))_{jj}, 1 \le i,j \le s(l,n)\}.
\end{eqnarray}


The next theorem provides a way to retrieve  an optimal solution of the problem~\eqref{P} from its   SOCP relaxation problem.
\begin{theorem}\label{th:9-recover} {\bf (Retrieval of  solutions from exact SOCP relaxation)}
Let $f,-g_i$ be SOCP-convex polynomials with an even degree $d$ and let $l=\frac{d}{2}$.
For problem~\eqref{P}, let $ {\rm Argmin}\eqref{P}\neq\emptyset$ and let $\tilde x\in\mathbb{R}^n$ be such that \begin{align}\label{S-retrival} g_i(\tilde x)>0, i=1,\ldots,m.\end{align}Let $y^*:=(y^*_\alpha)\in \mathbb{R}^{s(d,n)}$ be an optimal solution of  problem~\eqref{SOCP*} and denote $x^*:=(L_{y^*}(x_1),\ldots, L_{y^*}(x_n))\in\mathbb{R}^n,$ where $x_i$ denotes the polynomial which maps a vector $x\in\mathbb{R}^n$ to its $i$th coordinate. Then, $x^*$ is an optimal solution of problem~\eqref{P}.
\end{theorem}
\begin{proof} Since the problem~\eqref{P} is convex and the Slater condition~\eqref{S-retrival} is valid, the KKT condition holds at an optimal solution of \eqref{P}. 
 Invoking Theorem~\ref{th:9}, we first see that  ${\rm val}\eqref{P}={\rm val}\eqref{SOCP-Recover}$, where $\eqref{SOCP-Recover}$ is  the following  SOCP relaxation problem \begin{align}\label{SOCP-Recover}\notag
 \max_{\substack{\mu\in \mathbb{R}, \lambda_i \ge 0,\\ M^{ij} \in S^{s(l,n)}}}\{\mu\mid & f_{\alpha}-\sum_{i=1}^m\lambda_i (g_i)_{\alpha}-\mu=\sum_{\substack{1 \le \beta,\gamma \le s(l,n) \\ i(\alpha)=i(\beta)+i(\gamma)}} \sum_{1 \le i,j \le s(l,n)} M^{ij}_{\beta \gamma}, \nonumber \\
&    M^{ij}_{\beta \gamma}=0,\, \forall \ (\beta,\gamma) \notin \{(i,i),(j,j),(i,j),(j,i)\}, \nonumber  \\
& \|\bigg(\begin{array}{c}    2 M^{ij}_{ij}\\
                                                        M^{ij}_{ii}-M^{ij}_{jj}
                                                                                  \end{array}\bigg)\| \le M^{ij}_{ii}+M^{ij}_{jj},  1 \le i,j \le s(l,n)\}. \tag{SOCP}
                                                                                  \end{align}
Since the problem~\eqref{SOCP*} is the Lagrangian dual problem of $\eqref{SOCP-Recover}$, it holds by weak duality that ${\rm val}\eqref{SOCP-Recover}\le {\rm val}\eqref{SOCP*}$ and hence, ${\rm val}\eqref{P}\le {\rm val}\eqref{SOCP*}$.  
{We next show that ${\rm val}\eqref{P}= {\rm val}\eqref{SOCP*}$. 
To do this, take} $\bar x:=(\bar x_1,\ldots,\bar x_n)\in  {\rm Argmin}\eqref{P}$ and denote $\bar y:=\bar x^{(d)}=(1, \bar x_1,\ldots,\bar x_n,  \bar x_1^2,\bar x_1\bar x_2,\ldots,\bar x_2^2,\ldots,\bar x_n^2,\ldots,\bar x_1^{d},\ldots,\bar x_n^d)$. Then, $\bar y_1:=1$ and $$ L_{\bar y}(g_i)=\sum_{\alpha=1}^{s(d,n)}(g_{i})_{\alpha} \bar y_\alpha=\sum_{\alpha=1}^{s(d,n)}(g_{i})_{\alpha} \bar x^{(d)}_\alpha=g_i(\bar x)\ge 0, i=1,\ldots, m.$$ Moreover,  from the definition of 
the moment matrix (see \eqref{M-2} and \eqref{M-1}), we have $${\bf M}_l(\bar y)=\sum_{\alpha=1}^{ s(d,n)} \bar y_\alpha M_\alpha= \sum_{\alpha=1}^{ s(d,n)} \bar x^{(l)}_\alpha M_\alpha =\bar x^{(l)}(\bar x^{(l)})^T\succeq 0,$$ which guarantees that $$\|\bigg(\begin{array}{c}    2({\bf M}_l(\bar y))_{ij}\\
                                                      ({\bf M}_l(\bar y))_{ii}-({\bf M}_l(\bar y))_{jj}
                                                                                  \end{array}\bigg)\| \le ({\bf M}_l(\bar y))_{ii}+({\bf M}_l(\bar y))_{jj}, 1 \le i,j \le s(l,n).$$ Hence, $\bar y$ is a feasible point of    problem~\eqref{SOCP*} and it in turn implies that $${\rm val}\eqref{SOCP*}\le L_{\bar y}(f)=\sum_{\alpha=1}^{s(d,n)}f_{\alpha} \bar y_\alpha=\sum_{\alpha=1}^{s(d,n)}f_{\alpha} \bar x^{(d)}_\alpha=f(\bar x)={\rm val}\eqref{P}.$$

So, we arrive at the conclusion that \begin{align}\label{2.3}{\rm val}\eqref{P}= {\rm val}\eqref{SOCP*}.\end{align}

Now,  let $y^*:=(y^*_\alpha)\in \mathbb{R}^{s(d,n)}$  be an optimal solution of  problem~\eqref{SOCP*}.  Then,  \begin{align}\label{}&{\rm val}\eqref{SOCP*}=L_{y^*}(f),\label{2.2} \\& y^*_1=1,\label{1.13-a}\\
& L_{y^*}(g_i)\ge 0, i=1,\ldots, m,\label{2.1} \\& \|\bigg(\begin{array}{c}    2({\bf M}_l(y^*))_{ij}\\
                                                      ({\bf M}_l(y^*))_{ii}-({\bf M}_l(y^*))_{jj}
                                                                                  \end{array}\bigg)\| \le ({\bf M}_l(y^*))_{ii}+({\bf M}_l(y^*))_{jj}, \, 1 \le i,j \le s(l,n).\label{1.14-a}
\end{align}
Under the validation of \eqref{1.13-a} and \eqref{1.14-a}, applying Jensen's inequality for SOCP-convex polynomials (see Proposition~\ref{prop:Jensen}), we obtain from \eqref{2.1} that $$0\ge L_{y^*}(-g_i)\ge -g_i(L_{y^*}(x_1),\ldots, L_{y^*}(x_n))=-g_i(x^*), i=1,\ldots, m,$$ which shows that $x^*$ is a feasible point of problem~\eqref{P} and thus, ${\rm val}\eqref{P}\le f(x^*).$  Similarly, we derive from \eqref{2.3} and  \eqref{2.2} and the Jensen's inequality that $${\rm val}\eqref{P}= L_{y^*}(f)\ge f(L_{y^*}(x_1),\ldots, L_{y^*}(x_n))=f(x^*).$$ So,  ${\rm val}\eqref{P}= f(x^*)$ and  $x^*$ is an optimal solution of problem~\eqref{P}. The proof is complete.
\end{proof}

We now provide a simple example to illustrate the preceding theorem on recovering an optimal solution of an SOCP-convex polynomial optimization problem from its SOCP relaxation.
\begin{example}
Consider the following SOCP-convex polynomial optimization problem
\begin{eqnarray*}
(EP_1) &\min & x_1^4-x_2 \\
& \mbox{ s.t. } & 1-x_1^4-x_2^4 \ge 0.
\end{eqnarray*}
Direct verification shows that $(0,{1})$ is an optimal solution of $(EP_1)$ with optimal value $-1$.

The SOCP relaxation problem of $(EP_1)$ is
\[
 \max_{\substack{\mu\in \mathbb{R}, \lambda \ge 0, \\
\sigma\in {\bf SDSOS}_{2,2}}}\{\mu \mid x_1^4-x_2- \lambda (1-x_1^4-x_2^4)-\mu =\sigma\}
\]
and the Lagrangian dual of the SOCP relaxation can be formulated as
\begin{eqnarray*}
(SOCP_{EP1}^*) & \inf_{y \in \mathbb{R}^{15}}& y_{11}-y_3,  \\
& & 1-y_{11}-y_{15} \ge 0, \\
& & y_1=1, \\
& & {\bf M}_l(y)=\left(\begin{array}{cccccc}
y_1 & y_2 & y_3 & y_4 & y_5 & y_6 \\
y_2 & y_4 & y_5 & y_7 & y_8 & y_9 \\
y_3 & y_5 & y_6 & y_8 & y_9 & y_{10} \\
y_4 & y_7 & y_8 & y_{11} & y_{12} & y_{13} \\
y_5 & y_8 & y_9 & y_{12} & y_{13} & y_{14} \\
y_6 & y_9 & y_{10} & y_{13} & y_{14} & y_{15}
                        \end{array}
 \right), \\
&& \|\bigg(\begin{array}{c}    2({\bf M}_l(y))_{ij}\\
                                                      ({\bf M}_l(y))_{ii}-({\bf M}_l(y))_{jj}
                                                                                  \end{array}\bigg)\| \le ({\bf M}_l(y))_{ii}+({\bf M}_l(y))_{jj}, \, 1 \le i,j \le 6.
\end{eqnarray*}
Note that
\begin{eqnarray*}
& & \|\bigg(\begin{array}{c}  2({\bf M}_l(y))_{ij}\\
                                                      ({\bf M}_l(y))_{ii}-({\bf M}_l(y))_{jj}
                                                                                  \end{array}\bigg)\| \le ({\bf M}_l(y))_{ii}+({\bf M}_l(y))_{jj} \\  & \Leftrightarrow & ({\bf M}_l(y))_{ii}\ge 0,  ({\bf M}_l(y))_{jj}\ge 0 {\mbox{ and }} \sqrt{({\bf M}_l(y))_{ii} ({\bf M}_l(y))_{jj}}\ge  ({\bf M}_l(y))_{ij}.
\end{eqnarray*}
Note that $(SOCP_{EP1}^*)$ is also a second-order cone program. Solving $(SOCP_{EP1}^*)$ via the conic programming solver MOSEK gives an optimal solution
${\bf y}^*=(1,0,1,0,0,1,0,0,0,1,0,0,0,0,1) \in \mathbb{R}^{15}$. The optimality of ${\bf y^*}$ can also be verified independently as follows.
For any feasible point of $(SOCP_{EP1}^*)$, one has $y_1=1$, $y_{11} \ge 0$, $y_{15} \ge 0$, $1-y_{11}-y_{15} \ge 0$ and
\[
y_6=y_1y_6 \ge y_3^2 \mbox{ and } y_{15}=y_1y_{15} \ge y_6^2.
\]
This implies that $0 \le y_{15} \le 1$ and $y_{15} \ge y_6^2 \ge y_3^4$ (and so, $y_3 \le 1)$. So, $y_{11}-y_3 \ge 0-1=-1$, and hence ${\rm val}(SOCP_{EP1}^*) \ge -1$. On the other hand,  it is easy to see that ${ y}^*=(1,0,1,0,0,1,0,0,0,1,0,0,0,0,1) \in \mathbb{R}^{15}$ is feasible for $(SOCP_{EP1}^*)$ with objective value $-1$. So, ${\bf y}^*$ is a global solution of $(SOCP_{EP1}^*)$. Now, applying the preceding theorem, gives us that the optimal solution is $x^*=(L_{{ y}^*}(x_1),L_{{y}^*}(x_2))=(0,1)$, which agrees with the true optimal solution.
\end{example}

\section{Applications to Conic-Convex Semi-Algebraic Programs}

In this section, we propose a conic linear  programming hierarchy and establish its convergence and present conditions for finite convergence at one-step of the hierarchy for  the conic convex semi-algebraic problem:
 \begin{align}\label{CP}
  \inf_{x\in \R^{n}}{\{f(x) \mid  x\in K, \; G(x)\in S\}},\tag{CP}
\end{align}
where $f:\R^n\to\R$ is a convex polynomial, $S\subset\mathbb{R}^p$ is a closed convex cone, $K\subset\R^n$ is the basic semi-algebraic set given by \begin{align}\label{O-set}K:=\{x\in\R^n\mid  g_i(x)\ge 0,\; i=1,\ldots, m\}\end{align} for some (not necessarily convex) polynomials $g_i:\R^n\to \R, i=1,\ldots, m,$   and  $G:\R^n\to \mathbb{R}^p$ is an $S$-concave  polynomial in the sense that $$\alpha G(x)+(1-\alpha)G(y)-G(\alpha x+(1-\alpha)y)\in -S\;\mbox{ for all }\; x,y\in\R^n, \alpha\in [0,1].$$


 Let $n_1, n_2\in\mathbb{N}$, $0\le n_1, n_2\le n, n=n_1+n_2$ and fix a positive even number ${r} \in \mathbb{N}$.
 We now define a hierarchy of conic programming relaxations  for the polynomial program~\eqref{CP}  as follows: for $k \in \mathbb{N}$,
\begin{align}\label{CRP} \sup_{\substack{\mu \in \mathbb{R}, c_{\bf p,q} \ge 0, \\  \lambda\in S^*,
\sigma_1\in \Sigma^2_{n_1,r}, \sigma_2\in {\bf SDSOS}_{n_2,r}}}\left\{\mu\mid f-{\langle \lambda, G\rangle}-\sum_{{\bf p,q} \in {(\mathbb{N}_0)^m}, |{\bf p}|+|{\bf q}| \le k}c_{{\bf p, q}}\prod_{i=1}^{m} \widehat{g}_i^{p_i}(1-\widehat{g}_i)^{q_i}-\mu =\sigma_1+\sigma_2\right\}, \tag{\rm CRP$_{k}^{r}$}\end{align} where $S^*:=\{v\in \mathbb{R}^p\mid v^Ts\ge 0 \mbox{ for all } s\in S\}$ is the dual cone of $S$.

In the following theorem we present a convergence of the above hierarchy using  a hyperplane separation theorem and the main theorem of  Section 3.
\begin{theorem}{\bf (Convergent conic linear  programming hierarchy for \eqref{CP})} \label{thm:1-CP-cone}
For problem~\eqref{CP}, let $M$ be a positive number such that {{$\displaystyle M > \max_{1 \le i\le m}\sup_{x \in K}\{g_i(x)\}$}} and denote $\widehat{g}_i(x)=\frac{g_i(x)}{M}$.   Let $n_1, n_2\in\mathbb{N}$,  $0\le n_1, n_2\le n, n=n_1+n_2$ and fix a positive even number ${r} \in \mathbb{N}$.
Suppose that Assumption A holds and that $K$ is a compact convex set. Assume further  that there exists $\tilde x\in K$ such that \begin{align}\label{Slater}G(\tilde x)\in {\rm int}S.\end{align}Then, ${\rm val}\eqref{CRP} \le {\rm val} ({\rm CRP}_{k+1}^r) \le {\rm val}\eqref{CP}$ for all $k\in\mathbb{N}$, and
\[
\lim_{k \rightarrow \infty} {\rm val}\eqref{CRP} ={\rm val}\eqref{CP}.
\]
\end{theorem}
\begin{proof}
Let $\bar x\in {\rm Argmin}\eqref{CP}$ and define  \begin{align*}\Omega:= \{(r,y)\in\R\times \mathbb{R}^{p}\mid&\exists x\in K,\; f(x)-f(\bar x)< r,\; y+G(x)\in S\}.\end{align*} Then, $\Omega\neq\emptyset$ due to $(f(\tilde x)-f(\bar x)+\epsilon,0)\in \Omega$ for each $\epsilon>0$. It is easy to check that $\Omega$ is a convex set and, as  $\bar x$ is an optimal solution of~\eqref{CP}, it follows that $(0,0)\notin \Omega$. Using  a separation theorem (see, e.g., \cite[Theorem~2.5]{Mor-Nam-14}), we find
  $ (\lambda_0,\lambda)\in(\R\times \mathbb{R}^p)\setminus\{0\}$ such that
\begin{align}\label{1.4}\inf{\Big\{\lambda_0 r+\langle \lambda, y\rangle \mid (r,y)\in \Omega\Big\}}\ge 0.\end{align} This ensures that $\lambda_0 \ge 0, \lambda\in S^*$.

 Let $\epsilon>0.$ Then, due to $\big(f(x)-f(\bar x)+\epsilon,-G(x)\big)\in \Omega$ for each $x\in K$, we derive from \eqref{1.4} that
\begin{align}\label{2.3-L-2-Ex}\lambda_0\big(f(x)-f(\bar x)+\epsilon\big)-{\langle \lambda, G\rangle}(x)\ge 0\;\mbox{ for all }\; x\in K.\end{align} Now, it follows from \eqref{Slater} and \eqref{2.3-L-2-Ex} that  $\lambda_0\neq 0$ and therefore, there is no loss of generality in assuming  that $\lambda_0=1.$ So, we obtain that $f(x)-f(\bar x)+\epsilon-{\langle \lambda, G\rangle}(x)\ge 0$ for all $x\in K.$ Since $\epsilon>0$ was arbitrarily chosen, we arrive at the conclusion that $ f(x)-f(\bar x)-{\langle \lambda, G\rangle}(x) \ge 0$ for all $x\in K.$

Let the function $h:\R^n\to \R$ be given by $h(x):=f(x)-{\langle \lambda, G\rangle}(x)$ for  $x\in\R^n.$
We see that $h$ is a convex polynomial and $h(x)\ge f(\bar x)$ for all $x\in K.$ It entails  especially  that ${\langle \lambda, G\rangle}(\bar x)=f(\bar x)-h(\bar x)\le 0.$ In addition, ${\langle \lambda, G\rangle}(\bar x)\ge 0$ as $\bar x$ is a feasible point of problem~\eqref{P}. So, $h(\bar x)=f(\bar x)\le h(x)$ for all $x\in K.$ In other words, $\bar x$ is an optimal solution of the following convex program \begin{align}\label{AP}\inf\limits_{x\in\R^n}{\{h(x)\mid  g_i(x)\ge 0,\; i=1,\ldots, m\}}.\tag{AP}\end{align}
Using Assumption A, we obtain from Theorem~\ref{thm:1}  that  ${\rm val}\eqref{CRP} \le {\rm val} ({\rm CRP}_{k+1}^r) \le {\rm val}\eqref{AP}$ for all $k\in\mathbb{N}$, and
\[
\lim_{k \rightarrow \infty} {\rm val}\eqref{CRP} ={\rm val}\eqref{AP}.
\]  Since ${\rm val}\eqref{AP}=f(\bar x)={\rm val}\eqref{CP},$ the proof of the theorem is complete. \end{proof}

\medskip


Note that for $\eqref{CP}$ if we let $S:=  S^p_+$ and assume that $G$ is an $S^p_+$-concave matrix polynomial then under the assumption of  the preceding theorem, we obtain that
\[
\lim_{k \rightarrow \infty} {\rm val}\eqref{SDR} ={\rm val}\eqref{CP},
\]
where \eqref{SDR} is the following SDP relaxation
\begin{align}\label{SDR} \sup_{\substack{\mu \in \mathbb{R}, c_{\bf p,q} \ge 0, \\  \lambda\in S^p_+,
\sigma_1\in \Sigma^2_{n_1,r}, \sigma_2\in {\bf SDSOS}_{n_2,r}}}\left\{\mu\mid f-{\rm Tr}(\lambda G)-\sum_{{\bf p,q} \in {(\mathbb{N}_0)^m}, |{\bf p}|+|{\bf q}| \le k}c_{{\bf p, q}}\prod_{i=1}^{m} \widehat{g}_i^{p_i}(1-\widehat{g}_i)^{q_i}-\mu =\sigma_1+\sigma_2\right\} \tag{\rm SDR$_{k}^{r}$}\end{align}
with  $n_1, n_2\in\mathbb{N}$, $0\le n_1, n_2\le n, n=n_1+n_2$ and ${r} \in \mathbb{N}$  a positive even number.

\medskip
In the case where the semi-algebraic set $K$ given in \eqref{O-set} is described by convex polynomial inequalities we show that the values of the relaxation problems converge to the common value of \eqref{CP} and its Lagrangian dual under a constraint qualification.

\begin{corollary}{\bf (Convergence to common primal-dual value)}
For problem~\eqref{CP}, let $M$ be a positive number such that {{$\displaystyle M > \max_{1 \le i\le m}\sup_{x \in K}\{g_i(x)\}$}} and denote $\widehat{g}_i(x)=\frac{g_i(x)}{M}$.   Let $n_1, n_2\in\mathbb{N}$, $0\le n_1, n_2\le n, n=n_1+n_2$ and fix a positive even number ${r} \in \mathbb{N}$.
Suppose that Assumption A holds and that $-g_i, i=1,\ldots, m,$ are convex polynomials. Assume further  that there exists $\tilde x\in \mathbb{R}^n$ such that \begin{align}&g_i(\tilde x)>0, \, i=1,\ldots, m,\label{sla-add}\\& G(\tilde x)\in {\rm int}S.\label{Slater-SOP}\end{align}Then, ${\rm val}\eqref{CRP} \le {\rm val} ({\rm CRP}_{k+1}^r) \le {\rm val}\eqref{CP}$ for all $k\in\mathbb{N}$, and
\[
\lim_{k \rightarrow \infty} {\rm val}\eqref{CRP} ={\rm val}\eqref{CP}=\max_{\substack{\lambda\in S^*,  \lambda_i \ge 0}}\inf\limits_{x\in\R^n}{\{f(x)-{\langle \lambda, G\rangle}(x)-\sum_{i=1}^m\lambda_i g_i(x)\}}.
\]
\end{corollary}
\begin{proof} Let $\bar x\in {\rm Argmin}\eqref{CP}$.  Arguing similarly as in the proof of Theorem~\ref{thm:1-CP-cone}, under \eqref{Slater-SOP},  we find $\lambda\in S^*$ such that $\bar x$ is an optimal solution of
 the following convex program \begin{align}\label{AOP}\inf\limits_{x\in\R^n}{\{f(x)-{\langle \lambda, G\rangle}(x)\mid  g_i(x)\ge 0,\; i=1,\ldots, m\}}.\tag{AOP}\end{align} This shows in particular that ${\rm val}\eqref{CP}=f(\bar x)={\rm val}\eqref{AOP}.$
 Moreover, by \eqref{sla-add},  employing  the Lagrangian duality for convex programs (see e.g., \cite[Theorem~3.1]{jeya-08}) applied to the problem~\eqref{AOP}, we find $\lambda_i\ge 0, i=1,\ldots, m$ such that $\inf\limits_{x\in\R^n}{\{f(x)-{\langle \lambda, G\rangle}(x)-\sum_{i=1}^m\lambda_i g_i(x)\}}={\rm val}\eqref{AOP}.$ Hence, $${\rm val}\eqref{AOP}=\max_{\substack{\lambda\in S^*,  \lambda_i \ge 0}}\inf\limits_{x\in\R^n}{\{f(x)-{\langle \lambda, G\rangle}(x)-\sum_{i=1}^m\lambda_i g_i(x)\}}.$$ Now, the proof is completed by invoking Theorem~\ref{thm:1-CP-cone}.\end{proof}

\end{document}